\documentclass[12pt]{article}
\usepackage[margin=1.27in]{geometry}
\usepackage[utf8]{inputenc}
\usepackage{amsmath,amsfonts,amssymb,amsthm,bbm}
\usepackage{xcolor,enumitem,tikz,hyperref}

\newtheorem{theorem}{Theorem}[section]

\newtheorem{prop}[theorem]{Proposition}
\newtheorem{lemma}[theorem]{Lemma}
\newtheorem{corollary}[theorem]{Corollary}

\usepackage{blkarray}

\theoremstyle{definition}
\newtheorem{definition}[theorem]{Definition}
\newtheorem{example}[theorem]{Example}
\newtheorem{remark}[theorem]{Remark}

\newcommand{\PP}{\mathbb{P}}
\newcommand{\CC}{\mathbb{C}}
\newcommand{\QQ}{\mathbb{Q}}
\newcommand{\RR}{\mathbb{R}}
\newcommand{\BB}{\mathcal{B}}
\newcommand{\II}{\mathcal{I}}
\newcommand{\VV}{\mathcal{V}}
\newcommand{\HH}{\mathcal{H}}
\newcommand{\kk}{\mathbbm{k}}
\newcommand{\ww}{\mathbf{w}}
\newcommand{\xx}{\mathbf{x}}
\newcommand{\uu}{\mathbf{u}}
\newcommand{\vv}{\mathbf{v}}

\newcommand{\pp}{\mathbf{p}}
\newcommand{\rr}{\mathbf{r}}
\renewcommand{\tt}{\mathbf{t}}

\renewcommand{\ss}{\mathbf{s}}
\newcommand{\qq}{\mathbf{q}}
\DeclareMathOperator{\rk}{rank}
\DeclareMathOperator{\sgn}{sgn}
\DeclareMathOperator{\supp}{supp}

\title{The slack realization space of a matroid}
\author{Madeline Brandt and Amy Wiebe}

\begin{document}

\maketitle

\begin{abstract}
We introduce a new model for the realization space of a matroid, which is obtained from a variety defined by a saturated determinantal ideal, called the slack ideal, coming from the vertex-hyperplane incidence matrix of the matroid. This is inspired by a similar model for the slack realization space of a polytope. We show how to use these ideas to certify non-realizability of matroids, and describe an explicit relationship to the standard Grassmann-Pl\"ucker realization space model. We also exhibit a way of detecting projectively unique matroids via their slack ideals by introducing a toric ideal that can be associated to any matroid. 
\end{abstract}

\section{Introduction}

Realization spaces of matroids are well studied objects \cite{BLSWZ, weird_bernd_book, mnev} which encode not only whether or not the matroid is realizable, but also carry additional information about the structure of the matroid. A realization (or representation) of a rank $d+1$ matroid $M$ is a set of vectors in $\kk^{d+1}$ which captures its independence structure. Roughly speaking, a realization space is the set of all such choices of vectors. 
Fundamental questions in the study of realization spaces of matroids include discovering whether or not a given matroid is realizable, determining over which field it is realizable, finding the structure of the set of realizations, and characterizing when realizations exist.
A celebrated theorem of Mn\"ev states that every semialgebraic set defined over the integers is stably equivalent to the realization space of some oriented matroid. That is, realization spaces of matroids can become arbitrarily complicated. In light of this, we aim to connect the combinatorics of the matroid to properties of its realization space. 
%\comment{Something about how we need more tools to study them because they are hard.}

We generalize a construction of \cite{slack_paper} in which they describe a model for the realization space of a polytope using the {\em slack matrix} of the polytope. This model gave a new framework for answering questions about the realizability of polytopes. We extend these results to the setting of matroids, creating the beginnings of a dictionary between the combinatorial properties of the matroid and the algebraic description of its realization space. 

% \subsection{Organization}
 
 In Section~\ref{SEC:bg} we introduce the main objects of study, as well as preliminary results and notation. In Section~\ref{SEC:realsp} we discuss two models for the realization space of a matroid. One of our main theorems, Theorem~\ref{THM:mothervariety}, shows how the two realization space models can be described via a single overarching variety. In Section~\ref{SEC:nonreal} we show how the slack realization model can be used to determine whether a matroid has a realization over a certain field. We also reframe the tools of final polynomials \cite{weird_bernd_book} in terms of slack ideals, and show how they can be used to improve computational efficiency of realizability checking.  In Section~\ref{SEC:toric} we introduce a toric ideal associated to a matroid and study its relationship to the projective uniqueness of the matroid. In Appendix \ref{AP:notation} we include a table of notation used throughout the paper. 
The computations in this paper are done in \texttt{Macaulay2} \cite{M2} with the help of the \texttt{Matroids} package \cite{M2matroids}; the code we used can be found at \href{http://sites.math.washington.edu/~awiebe}{\texttt{http://sites.math.washington.edu/$\sim$awiebe}}.

%==============================================================================
\section{The slack matrix of a matroid}
\label{SEC:bg}

Much of this section is analogous to \cite[\S2]{slack_paper} to which we refer the reader for further details and excluded proofs. 
We assume the reader has familiarity with the basic definitions from matroid theory, see \cite{Oxley} or \cite{GM}. Throughout the paper, we assume all matroids are simple (having no loops or parallel elements).

Let $\kk$ be a field. Let $M = (E,\BB)$ be a matroid of rank $d+1$ with ground set $E = \{\vv_1,\ldots, \vv_n\}$, where each $\vv_i\in\kk^{d+1}$ and $\BB$ is its set of bases. If $V$ is the matrix with columns $\vv_1,\ldots, \vv_n$, then the independent sets of $M$ are the linearly independent subsets of columns, and we write $M = M[V]$. %Throughout the paper we assume that all matroids are simple. i.e., no loops or parallel elements.

Let $\HH(M)$ denote the set of hyperplanes of $M$, which are the closed subsets (flats) of rank $d$. In $M[V]$, each hyperplane $H\in\HH(M)$ corresponds to a linear subspace of $\kk^{d+1}$, so is determined by some linear equation; that is,
$H = \{\xx\in E: \alpha_H^\top\xx = 0\}$.
For $\HH(M) = \{H_1,\ldots, H_h\}$, let $W$ be the matrix whose columns are the hyperplane defining normals $\alpha_1, \ldots, \alpha_h$, or some multiple, $\lambda_j\alpha_j$ for $\lambda_j\in\kk^*$, thereof.

\begin{definition}
The {\em slack matrix} of the matroid $M = M[V]$ over $\kk$ is the $n\times h$ matrix $S_{M[V]} = V^\top W.$ %This depends on the realization $V$.
\label{DEFN:slackmatrix}
\end{definition}
%Observe for any invertible diagonal matrix $D\in\kk^{h\times h}$, the matrix $S_MD$ is also a slack matrix of $M$.

We wish to parametrize the set of realizations of a matroid by its slack matrices. So, we must determine the characteristics which define the set of all possible slack matrices of a given matroid. We begin by considering the rank of a slack matrix. 

\begin{lemma} If $S$ is a matrix having the same support as the slack matrix of some rank $d+1$ matroid $M = M[V]$, then $\rk(S)\geq d+1$. (See \cite[Lemma~3.1]{slack_paper}.) \label{LEM:trianglesubmat} \end{lemma}

\begin{corollary} If $M=M[V]$ is a rank $d+1$ matroid then $\rk(S_M) = d+1$. \label{COR:rank} \end{corollary}
\begin{proof} The factored form of $S_M \in \kk^{n\times (d+1)}\times\kk^{(d+1)\times h}$ implies that $\rk(S_M)\leq d+1$. The result then follows from Lemma~\ref{LEM:trianglesubmat}.
%Since $M$ is rank $d+1$, we have $\rk(V) = d+1$. It remains to show that $\rk(W) = d+1$. (True if $\HH$ is an essential hyperplane arrangement, i.e., $\cap_{H\in\HH} H = \{0\}$, which I think it has to be for a realizable matroid - use same argument as in Polytope paper Lemma 3.1 using lattice of flats?)
\end{proof}

Now, let $M = (E,\BB)$ be an abstract matroid of rank $d+1$. 
Unless otherwise stated, we take $E = [n]=\{1,\ldots, n\}$.
A {\em realization} of $M$ over $\kk$ is a collection of vectors $V = \{\vv_1,\ldots, \vv_n\} \subset \kk^\ell$ such that the independent subsets of $V$ are indexed by the independent sets of the matroid, so $M=M[V]$. %Equivalently, as in \cite[Chapter 8]{BLSWZ}, we can consider a realization as a map $\phi: E\to \kk^{d+1}$ such that $\det(\phi(e_0),\ldots, \phi(e_d)) = 0$ if and only if $\{e_0,\ldots, e_d\}\notin\BB$. Not all abstract matroids have such a representation. 
A matroid with a realization is called {\em realizable} (also {\em representable, linear} or {\em coordinatizable}). 

\begin{lemma} The rows of a slack matrix $S_M$ form a realization of the matroid $M$. \label{LEM:rowrealiz} \end{lemma}
\begin{proof} It suffices to show that if we label the rows of $S_M$ with $[n]$, the subsets indexing linearly independent rows of $S_M$ are the independent sets of $M$. 
Since $S_M = V^TW$, if a subset $J$ of $E$ is dependent, then there exists a vector $\beta\in\kk^n$ with support indexed by $J$ such that $V\beta = 0$. But now, $\beta^\top S_M = (V\beta)^\top W = 0$, so $J$ also indexes a dependent subset of the rows of $S_M$. 

Conversely, suppose $J$ indexes a dependent subset of the rows of $S_M$. Then for some $\beta\in\kk^n$ with support indexed by $J$, we have 
$0 = \beta^\top S_M = (V\beta)^\top W$. Since $W$ has full rank by Corollary~\ref{COR:rank}, it must be the case that $V\beta = 0$, so that $J$ also indexes a dependent set of $M$. \end{proof}

From now on we assume that realizations come with a fixed labelling of ground set elements and hyperplanes, so that two slack matrices of the same matroid cannot differ by permutations of rows and columns. This also allows us to identify hyperplanes of a realization by vectors or the indices of those vectors. 
Now, we characterize the set of matrices which correspond to slack matrices of a matroid $M$. 

\begin{theorem} Let $M$ be a rank $d+1$ matroid with $n$ elements and hyperplanes $\mathcal{H}(M) = \{H_1, \ldots, H_h\}$. A matrix $S \in \kk^{n \times h}$ is the slack matrix of some realization of~$M$ if and only if both of the following hold:
\begin{enumerate}[label=(\roman{enumi})]
\hspace{0.7 in}
\begin{minipage}{0.3 \textwidth}
\item{$\text{supp}(S) = \text{supp}(S_M)$}
\end{minipage}
\hspace{0.5 in}
\begin{minipage}{0.3 \textwidth}
\item{$\rk(S) = d+1$.}
\end{minipage}
\end{enumerate}
\label{THM:slackconditions}
\end{theorem}
\begin{proof}
Suppose $S$ is the slack matrix of some realization of $M$. Then (i) holds trivially, and (ii) holds by Corollary~\ref{COR:rank}. 

Conversely, suppose $S$ is a matrix satisfying (i) and (ii). By (ii), $S$ has some rank factorization $S=AB$, where $A\in\kk^{n\times(d+1)}$ and $B\in\kk^{(d+1)\times h}$. %is an $n\times(d+1)$ matrix and $B$ is a $(d+1)\times h$ matrix. 
Let $\mathbf{a}_1,\ldots, \mathbf{a}_n \in \kk^{d+1}$ be the rows of $A$ and $\mathbf{b}_1,\ldots, \mathbf{b}_h \in \kk^{d+1}$ be the columns of $B$. Then we claim that the rows of $A$ give a realization of $M$; that is $M = M[A^\top]$.
To see this, we show that the hyperplanes of $M$ are also hyperplanes of $M[A^\top]$, and that $M[A^\top]$ can not contain more hyperplanes.

By (i), for each hyperplane $H_j$ of $M$, there is a column of $S$ with zeros in positions indexed by elements of $H_j$. Since $S= AB$, we have $\mathbf{b}_{j}^\top \mathbf{a}_i = 0$ if and only if $i\in H_j$. Thus 
$$\{\xx\in\kk^{d+1} : \mathbf{b}_j^\top \xx = 0\}\cap \{\text{rows}(A)\} = \{\mathbf{a}_i\}_{i\in H_j}$$
so that $H_j$ is also a hyperplane of the matroid $M[A^\top]$.

%Every subset of $d$ independent elements $\{i_1, \ldots, i_d\}$ is contained in a unique hyperplane, namely its closure $\overline{\{i_1, \ldots, i_d\}}$. 
Now suppose $M[A^\top]$ has an extra hyperplane $H \not \in \mathcal{H}(M)$. Let $\{i_1, \ldots, i_d\} \subset H$ be any $d$ distinct independent elements of $H$. Since $i_1,\ldots, i_d$ are also elements of matroid $M$, the flat $\overline{\{i_1, \ldots, i_d\}}$ is a hyperplane $H'$ of $M$, and thus also a hyperplane of $M[A^\top]$, but $H' \neq H$. However, this means that $\{i_1, \ldots, i_d\}$ are contained in two distinct hyperplanes of $M[A^\top]$, which is not possible, so we arrive at a contradiction. 
\end{proof}

We now recall two equivalence relations on the set of realizations of a matroid $M$, and illustrate how these equivalences are reflected in slack matrices. For $A\in GL(\kk^{d+1})$, it is easy to check that $V$ and $AV$ define the same matroid. We call these realizations {\em linearly equivalent}.
% The proof of "one can show that" in the previous sentence. (deleted by Maddie)
%Let $A\in GL(\kk^{d+1})$. Notice that $V, AV$ define the same matroid as follows: $J$ is a dependent set of $M[V]$ if and only if there exists a vector $\beta\in\kk^n$ with support indexed by $J$ such that $V\beta = 0$, which happens if and only if $(AV)\beta = 0$, since $A$ is full rank. Thus $M[V],M[AV]$ have the same dependent sets, so they are the same matroid, as claimed. The realizations $V, AV$ are called {\em linearly equivalent}. 
If $P\in\kk^{n\times n}$ is a permutation matrix which sends $i\mapsto \sigma(i)$, then $V$ and $VP$ define the same matroid up to relabelling the ground set $E=[n]$ with $\sigma(1),\ldots, \sigma(n)$. Thus if $A\in GL(\kk^{d+1})$ and $B$ is a permutation matrix with any element of $\kk^*$ in the $1$'s positions, then $V, AVB$ define the same matroid. We call the realizations $V, AVB$ {\em projectively equivalent}. Call a matroid $M$ {\em projectively unique} (over $\kk$) if all realizations are projectively equivalent.

%Two matroids $M[V],M[U]$ are called {\em projectively equivalent} if $U = AVB$, where . This is equivalent to the definition of \cite[\S6.3]{Oxley}, and is so named because their realizations are related via an automorphism of the underlying projective space $\PP^d$.
%A matroid $M$ for which all realizations are projectively equivalent is called {\em projectively unique}. 

\begin{lemma} Two realizations of a matroid $M$ are projectively equivalent if and only if their slack matrices are the same up to row and column scaling. %and permutations. 
\label{LEM:pe} \end{lemma}

\begin{proof} Suppose we have projectively equivalent representations $U,V$ of $M$. Then $U = AVB$, where $A\in GL(\kk^{d+1})$ and without loss of generality $B$ is an invertible $n\times n$ diagonal matrix (since we have assumed a fixed labelling of our matroid).%a permutation simply relabels the ground set and the hyperplanes accordingly).

If $H=\{\vv_{i_1},\ldots, \vv_{i_k}\}$ be a hyperplane of $M[V]$, %defined by $\alpha_H\in\kk^{d+1}$.
 then $H' = \{\uu_{i_1},\ldots, \uu_{i_k}\}$ is a hyperplane of $M[U]$. Furthermore, if $\alpha_H\in\kk^{d+1}$ is normal to $H$, then since {${\uu_i = A\vv_i\cdot b_i}$}, $A^{-\top}\alpha_H$ is normal to $H'$, %we have $H' = \{A\xx \in\kk^{d+1} : \alpha_H^\top\xx=0\}\cap U$. Hence $H' = \{\yy\in\kk^{d+1} : \alpha_H^\top A^{-1}\yy=0\}\cap U$, 
 so that a slack matrix of $M[U]$ is
$$
S_{M[U]}  = U^\top \begin{bmatrix} A^{-\top}\alpha_H \end{bmatrix}_{H\in\HH} 
		 = B^\top V^\top A^\top\begin{bmatrix} A^{-\top}\alpha_H \end{bmatrix}_{H\in\HH} 
		 = B^\top V^\top W 
		 = B^\top S_{M[V]}.
$$
Since we can always scale columns of a slack matrix, this completes the proof. 

Conversely, suppose we have realizations $U$ and $V$ of the matroid $M$ such that 
${S_{M[U]} = D_n S_{M[V]} D_h}$ for invertible diagonal matrices $D_n\in\kk^{n\times n},D_h\in\kk^{h\times h}$.
By definition, $S_{M[U]} = U^\top W$ and $S_{M[V]} = V^\top W'$.  Multiplying both sides of the above equation on the right by $W^\top(WW^\top)^{-1}$, we find
$$
U^\top = D_n V^\top W' D_h W^\top(WW^\top)^{-1}.
$$
We see that $W' D_h W^\top(WW^\top)^{-1}$ is a $(d+1) \times (d+1)$ invertible matrix, which makes $U$ and $V$ projectively equivalent, as desired. \end{proof}

By taking $B,D_n$ each to be the $n\times n$ identity matrix in the above proof, we recover the following lemma. 
\begin{lemma} 
\label{LEM:columnscale}
Two realizations of a matroid $M$ are linearly equivalent if and only if their slack matrices are the same up to column scaling.
\end{lemma}

We now define an analog of the slack matrix which can be constructed for any abstract matroid, even ones which are not realizable, as follows. 

\begin{definition} \label{DEFN:symbslack}
Define the {\em symbolic slack matrix} of matroid $M$ to be the matrix $S_M(\xx)$ with rows indexed by elements $i\in E$, columns indexed by hyperplanes $H_j\in \HH(M)$ and $(i,j)$-entry
$$S_M(\xx)_{ij} = \begin{cases} x_{ij} &\text{ if } i\notin H_j \\ 0 &\text{ if } i\in H_j. \end{cases}$$
% This matrix can be obtained by replacing each nonzero entry of $S_M$ by a distinct variable $x_{ij}$.
 The {\em slack ideal} of $M$ is the saturation of the ideal generated by the $(d+2)$-minors of $S_M(\xx)$, namely
\begin{align*}
I_M :&=\Big\langle (d+2)-\text{minors of }S_M(\xx)\Big\rangle { :} \left(\prod_{i=1}^n\prod_{j:i\not \in H_j} x_{ij}\right)^\infty 
&\subset \kk[\xx]. %= \kk[x_{ij}\ |\ 1 \leq i \leq n,\ 1 \leq j \leq h,\ i \not \in H_j].
\end{align*}
Suppose there are $t$ variables in $S_M(\xx)$. The {\em slack variety} is the variety $\VV(I_M)\subset \kk^t$. The saturation of $I_M$ by the product of all the variables guarantees that there are no components of $\VV(I_M)$ that live entirely in coordinate hyperplanes. If $\mathbf{s}\in\kk^t$ is a zero of $I_M$, then we identify it with the matrix $S_M(\mathbf{s})$. 
\end{definition}

\begin{example}
\label{EG:four_lines}
Consider the rank 3 matroid $M_4 = M[V]$ for $V$ whose columns are
%Consider the rank 3 matroid $M_4 = ([6],\BB)$ where $\BB$ consists of all 3 element sets except $123$, $246$, $345$, and $156$. This matroid is realizable in $\RR^3$ via the vectors 
$\vv_1 = (-2,-2,1)^\top$,
$\vv_2 = (-1,1,1)^\top$,
$\vv_3 = (0,4,1)^\top$,
$\vv_4 = (2,-2,1)^\top$,
$\vv_5 = (1,1,1)^\top$,
$\vv_6 = (0,0,1)^\top$.
Projecting onto the plane $z=1$, this can be visualized as the points of intersection of four lines in the plane, as in Figure \ref{FIG:four_lines}. %The four non-bases give us four of the hyperplanes, and in addition to this, there are three hyperplanes $25,14,36$. 
\begin{figure}
\begin{center}
\begin{minipage}{0.25 \textwidth}
\includegraphics[height = 2 in]{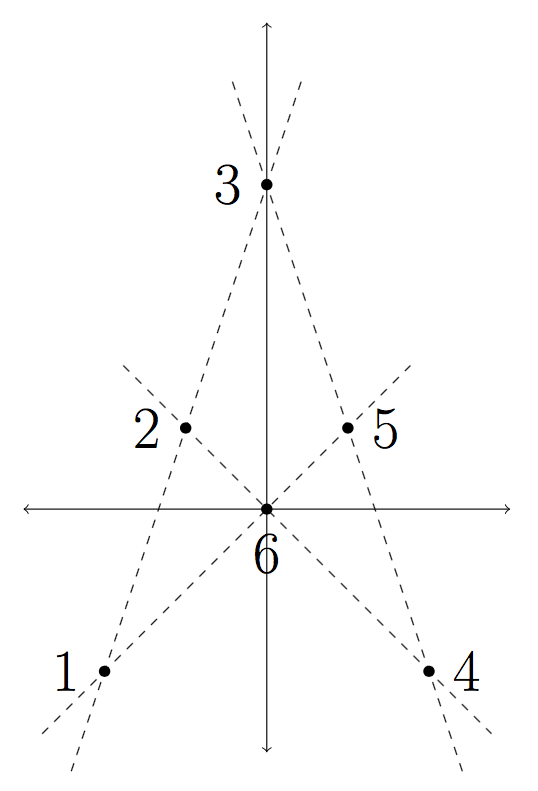}
\end{minipage}
\begin{minipage}{0.6 \textwidth}
\[%S_{M_4}(\xx) = 
\begin{blockarray}{cccccccc}
 &H_1 & H_2& H_3 & H_4 & H_5 & H_6 & H_7 \\
 &123 & 246 & 345 & 156 & 25 & 14 & 36\\
\begin{block}{c[ccccccc]}
1 & 0 & x_{12} & x_{13} & 0 & x_{15} &0 & x_{17} \\
2 & 0 & 0 & x_{23} & x_{24} &0 &x_{26} & x_{27} \\
3 & 0 & x_{32} & 0 & x_{34} &x_{35} & x_{36} &0 \\
4 & x_{41} & 0 & 0 & x_{44} &x_{45} & 0& x_{47} \\
5 & x_{51} & x_{52} & 0 & 0 & 0& x_{56} & x_{57} \\
6 & x_{61} & 0 & x_{63} & 0 &x_{65} & x_{66} & 0\\
\end{block}
\end{blockarray}
 \]
\end{minipage}
\caption{The point-line configuration of Example \ref{EG:four_lines}, and its symbolic slack matrix.}%$S_{M_4}(\xx)$.}
\label{FIG:four_lines}
\end{center}
\end{figure}
A slack matrix for this realization is then
{%\arraystretch{0.9}
\fontsize{10.5}{12}\selectfont
\begin{align*}
S_{M_4} &=
\begin{blockarray}{ccc}
\begin{block}{@{\;\;}[*{3}{@{\;}c@{\;}}]}
 -2 & -2 & 1 \\
 -1 & 1 & 1 \\
 0 & 4 & 1 \\
 2 & -2 & 1 \\
 1 & 1 & 1 \\
 0 & 0 & 1\\
\end{block}
\end{blockarray}\;\;
\begin{blockarray}{*{7}{@{\,}c@{\,}}}
 H_1 & H_2& H_3 & H_4 & H_5 & H_6 & H_7 \\
 \scriptstyle{123} &  \scriptstyle{246} &  \scriptstyle{345} &  \scriptstyle{156} &  \scriptstyle{25} &  \scriptstyle{14} &  \scriptstyle{36}\\
\begin{block}{[*{7}{@{\,}c@{\,}}]}
-3 & 3 & 6& -3 & 0 & 0 & 4\\
1 & 3 & 2 & 3 & 2 & 4 & 0 \\
-4 & 0 & -8 & 0 & -2 & 8 & 0\\
\end{block}
\end{blockarray}
&=
\begin{blockarray}{c*{7}{@{\,}c@{\,}}}
 &H_1 & H_2& H_3 & H_4 & H_5 & H_6 & H_7 \\
 & \scriptstyle{123} &  \scriptstyle{246} &  \scriptstyle{345} &  \scriptstyle{156} &  \scriptstyle{25} &  \scriptstyle{14} &  \scriptstyle{36}\\
\begin{block}{c@{\;\;}[*{7}{@{\,}c@{\,}}]@{\;}}
1& 0 & -12 & -24 & 0 & -6 & 0 & -8 \\
2& 0 & 0 & -12 & 6 & 0 & 12 & -4 \\
3& 0 & 12 & 0 & 12 & 6 & 24 & 0 \\
4& -12 & 0 & 0 & -12 & -6 & 0 & 8 \\
5& -6 & 6 & 0 & 0 & 0 & 12 & 4 \\
6& -4 & 0 & -8 & 0 & -2 & 8 & 0 \\
\end{block}
\end{blockarray},
\end{align*}
}
where using $\{\vv_{j_1}, \ldots, \vv_{j_d}\}\subset H$ independent, we calculate each $\alpha_H$ as 
\begin{equation}\label{EQ:pluckerdet}
\text{det}
\begin{bmatrix}
\widehat{e_1}& | & \cdots & |\\
\vdots & \vv_{j_1} & \cdots & \vv_{j_d}\\
\widehat{e_{d+1}} & | & \cdots & |\\
\end{bmatrix}.
\end{equation}

The symbolic slack matrix of $M_4$ is in Figure \ref{FIG:four_lines}.
We take the ideal of $4$-minors of this matrix, and saturate with respect to the product of all of the variables to get the slack ideal $I_{M_4}$. This has codimension~12, degree~293 and is generated by the 72~binomial generators in Table \ref{TAB:72things}. In Section \ref{SEC:toric} we will see these correspond to the 72 cycles in the bipartite non-incidence graph of this configuration (Figure \ref{FIG:fourlinesgraph}).

\begin{table}
\begin{center}
\footnotesize
\begin{tabular}{| l | l |}
\hline
$\!\!\deg 2\!\!$ & $x_{36}x_{65}+x_{35}x_{66},\ x_{26}x_{63}-x_{23}x_{66},\ x_{15}x_{63}-x_{13}x_{65},\ x_{56}x_{61}-x_{51}x_{66},\
x_{45}x_{61}-x_{41}x_{65},$\\ 
		& $x_{27}x_{56}+x_{26}x_{57},\ x_{36}x_{52}-x_{32}x_{56},\ x_{17}x_{52}-x_{12}x_{57},\ x_{47}x_{51}-x_{41}x_{57},\ x_{17}x_{45}+x_{15}x_{47},$\\ 
		& $x_{35}x_{44}-x_{34}x_{45},\ x_{27}x_{44}-x_{24}x_{47},\ x_{26}x_{34}-x_{24}x_{36},\ x_{15}x_{32}-x_{12}x_{35},\ x_{17}x_{23}-x_{13}x_{27}$\\
\hline
$\!\!\deg 3\!\!$ & $x_{47}x_{56}x_{65}-x_{45}x_{57}x_{66},\ x_{17}x_{56}x_{65}+x_{15}x_{57}x_{66},\ x_{12}x_{56}x_{65}+x_{15}x_{52}x_{66},\ x_{26}x_{47}x_{65}+x_{27}x_{45}x_{66},$\\ 
		&$x_{26}x_{44}x_{65}+x_{24}x_{45}x_{66},\ x_{17}x_{26}x_{65}-x_{15}x_{27}x_{66},\ x_{17}x_{56}x_{63}+x_{13}x_{57}x_{66},\ x_{12}x_{56}x_{63}+x_{13}x_{52}x_{66},$\\ 
		&$x_{27}x_{45}x_{63}+x_{23}x_{47}x_{65},\ x_{24}x_{45}x_{63}+x_{23}x_{44}x_{65},\ x_{12}x_{36}x_{63}+x_{13}x_{32}x_{66},\ x_{24}x_{35}x_{63}+x_{23}x_{34}x_{65},$\\
		&$x_{23}x_{57}x_{61}+x_{27}x_{51}x_{63},\ x_{15}x_{57}x_{61}+x_{17}x_{51}x_{65},\ x_{13}x_{57}x_{61}+x_{17}x_{51}x_{63},\ x_{35}x_{52}x_{61}+x_{32}x_{51}x_{65},$\\ 
		&$x_{15}x_{52}x_{61}+x_{12}x_{51}x_{65},\ x_{13}x_{52}x_{61}+x_{12}x_{51}x_{63},\ x_{26}x_{47}x_{61}+x_{27}x_{41}x_{66},\ x_{23}x_{47}x_{61}+x_{27}x_{41}x_{63},$\\ 
		&$x_{13}x_{47}x_{61}+x_{17}x_{41}x_{63},\ x_{36}x_{44}x_{61}+x_{34}x_{41}x_{66},\ x_{26}x_{44}x_{61}+x_{24}x_{41}x_{66},\ x_{23}x_{44}x_{61}+x_{24}x_{41}x_{63},$\\
		&$x_{35}x_{47}x_{56}+x_{36}x_{45}x_{57},\ x_{34}x_{47}x_{56}+x_{36}x_{44}x_{57},\ x_{17}x_{35}x_{56}-x_{15}x_{36}x_{57},\ x_{35}x_{47}x_{52}+x_{32}x_{45}x_{57},$\\ 
		&$x_{34}x_{47}x_{52}+x_{32}x_{44}x_{57},\ x_{27}x_{34}x_{52}+x_{24}x_{32}x_{57}, \ x_{13}x_{26}x_{52}+x_{12}x_{23}x_{56},\ x_{36}x_{45}x_{51}+x_{35}x_{41}x_{56},$\\ 
		&$x_{32}x_{45}x_{51}+x_{35}x_{41}x_{52},\ x_{12}x_{45}x_{51}+x_{15}x_{41}x_{52},\ x_{36}x_{44}x_{51}+x_{34}x_{41}x_{56},\ x_{32}x_{44}x_{51}+x_{34}x_{41}x_{52},$\\
		&$x_{26}x_{44}x_{51}+x_{24}x_{41}x_{56},\ x_{27}x_{36}x_{45}-x_{26}x_{35}x_{47},\ x_{17}x_{32}x_{44}+x_{12}x_{34}x_{47},\ x_{15}x_{23}x_{44}+x_{13}x_{24}x_{45},$\\ 
		&$x_{17}x_{26}x_{35}+x_{15}x_{27}x_{36},\ x_{13}x_{26}x_{35}+x_{15}x_{23}x_{36},\ x_{15}x_{27}x_{34}+x_{17}x_{24}x_{35},\ x_{15}x_{23}x_{34}+x_{13}x_{24}x_{35},$\\ 
		&$x_{17}x_{26}x_{32}+x_{12}x_{27}x_{36},\ x_{13}x_{26}x_{32}+x_{12}x_{23}x_{36},\ x_{17}x_{24}x_{32}+x_{12}x_{27}x_{34},\ x_{13}x_{24}x_{32}+x_{12}x_{23}x_{34}$\\
\hline
$\!\!\deg 4\!\!$ & $x_{27}x_{35}x_{52}x_{63}-x_{23}x_{32}x_{57}x_{65},\ x_{17}x_{36}x_{44}x_{63}-x_{13}x_{34}x_{47}x_{66},\ x_{24}x_{35}x_{57}x_{61}-x_{27}x_{34}x_{51}x_{65},$\\
		&$x_{23}x_{34}x_{52}x_{61}-x_{24}x_{32}x_{51}x_{63},\ x_{12}x_{36}x_{47}x_{61}-x_{17}x_{32}x_{41}x_{66},\ x_{13}x_{32}x_{44}x_{61}-x_{12}x_{34}x_{41}x_{63},$\\
		&$x_{15}x_{26}x_{44}x_{52}-x_{12}x_{24}x_{45}x_{56},\ x_{13}x_{26}x_{45}x_{51}-x_{15}x_{23}x_{41}x_{56},\ x_{12}x_{23}x_{44}x_{51}-x_{13}x_{24}x_{41}x_{52}$\\
\hline
\end{tabular}
\end{center}
\caption{The 72 generators of $I_{M_4}$.}
\label{TAB:72things}
\end{table}

 \end{example}
 
\begin{remark} In \cite{slack_paper}, given a set of $n$ vertices $V\subset\kk^d$ defining a $d$-polytope $P=\text{conv}(V)$, they include only the facet defining hyperplanes in the slack matrix. We can also form a matroid associated to this polytope by considering all the hyperplanes; that is, we define the matroid $M = M[V']$ where $V' \subset \kk^{(d+1)\times n} $ is the matrix obtained from $V$ by appending a 1 to each vector. Then the symbolic slack matrix of~$P$ defined in \cite{slack_paper} is the restriction of the symbolic slack matrix of matroid $M$ to the subset of columns corresponding to facet-defining hyperplanes. 
%Notice that the construction of this slack ideal is analogous to the construction of \cite{slack_paper} for $d$-polytopes. Given a set of $n$ vertices $V \subset \kk^d$, in the polytope case we consider only the facet defining hyperplanes, where as here we would consider all the hyperplanes which are defined by the matroid $M=M[V']$, where $V' \subset \kk^{(d+1)\times n} $ is the matrix obtained from $V$ by appending a 1 to each vector. This means if $P=\text{conv}(V)$, then 
Thus the slack ideal of the polytope is always contained in the slack ideal of the matroid, $I_P\subseteq I_M$. We illustrate with the following example.%\comment{say something more detailed about this or / and add an example}
\end{remark}
 
\begin{example} We consider the triangular prism $P$ labelled as in Figure \ref{FIG:toblerone}. As a 3-polytope, its facets are given by the hyperplanes $1234,1256,3456,135,246$ and the symbolic slack matrix is in Figure \ref{FIG:toblerone}. Its slack ideal $I_P$ is generated by 3 binomials.

\begin{figure}
\begin{minipage}{2.15in}
\includegraphics[width=1.8in]{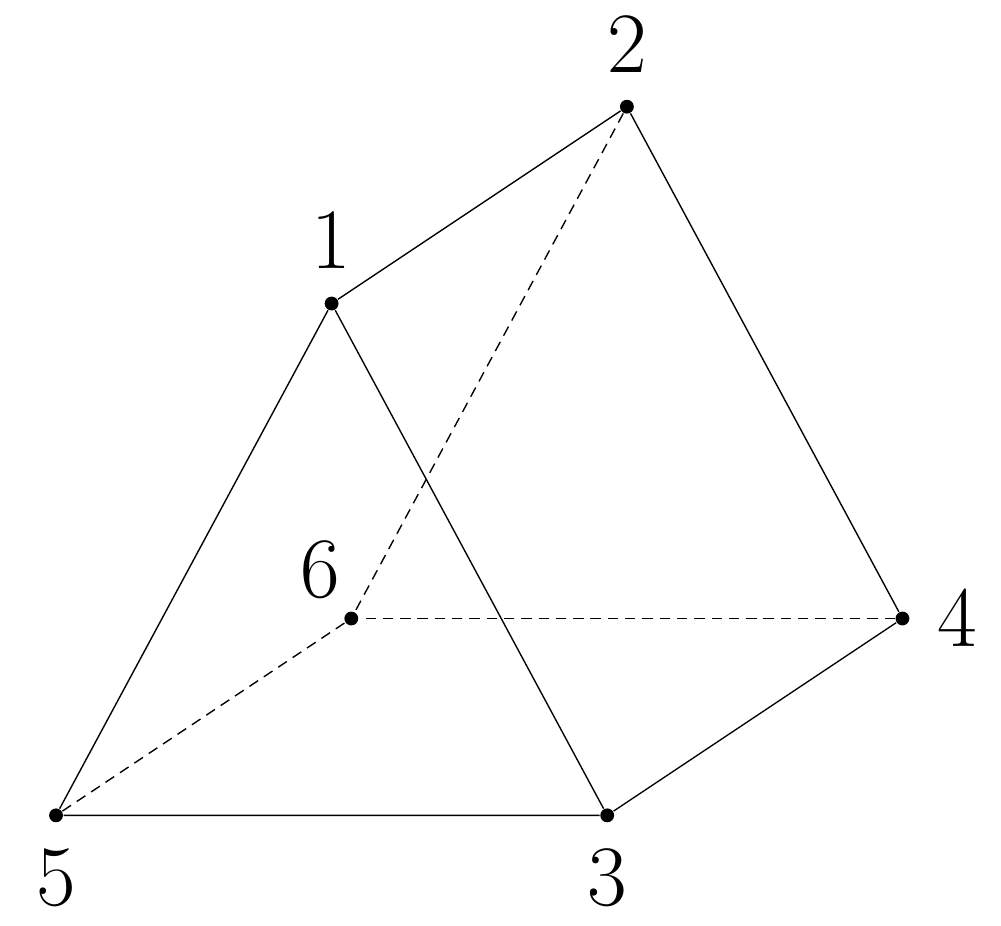}
\end{minipage}\begin{minipage}{3.5in}
\[ S_P(\xx)=
\begin{blockarray}{cccccc}
 	&H_1 	& H_2	& H_3 	& H_4 	& H_5 \\
	&1234 	& 1256 	& 3456 	& 135 	& 246 \\
\begin{block}{c[ccccc]}
1 &0			&0		&x_{13}	&0		&x_{15}	 \\
2 &0			&0		&x_{23}	&x_{24}	&0	 \\ 
3 &	0		&x_{32}	&	0	&	0	&x_{35}	 \\ 
4 &	0		&x_{42}	&	0	&x_{44}	&0	 \\  
5 &	x_{5,1}	&	0	&	0	&	0	&x_{55}	 \\ 
6 &	x_{6,1}	&	0	&	0	&x_{64}	&	0 \\ 
\end{block}
\end{blockarray}
 \]
 \end{minipage}
 \caption{The triangular prism and its slack matrix as a polytope.}
 \label{FIG:toblerone} 
 \end{figure}

Considering $P$ as a rank 4 matroid which has the 3 facets $1234$, $1256$, $3456$ of $P$ as its non-bases, we obtain following symbolic slack matrix.
\[
S_M(\xx)=
\begin{blockarray}{cccccc|cccccc}
 	&H_1 	& H_2	& H_3 	& H_4 	& H_5 	&H_6	&H_7	&H_8	&H_9	&H_{10}	&H_{11}\\
	&1234 	& 1256 	& 3456 	& 135 	& 246 	&136		& 146	&145		&245		&235		&236\\
\begin{block}{c@{\;}[*{5}{@{\;}c@{\;}}|@{\hspace{-3.3pt}}*{6}{@{\;}c@{\;}}]}
1 &0			&0		&x_{13}	&0		&x_{15}	&0		&0		&0		&x_{19}	&x_{1,10}	&x_{1,11}	\\
2 &0			&0		&x_{23}	&x_{24}	&0	 	&x_{26}	&x_{27}	&x_{28}	&0		&0		&0		\\
3 &	0		&x_{32}	&	0	&	0	&x_{35}	&0		&x_{37}	&x_{38}	&x_{39}	&0		&0		\\
4 &	0		&x_{42}	&	0	&x_{44}	&0	 	&x_{46}	&0		&0		&0		&x_{4,10}	&x_{4,11}	\\
5 &	x_{51}	&	0	&	0	&	0	&x_{55}	&x_{56}	&x_{57}	&0		&0		&0		&x_{5,11}	\\
6 &	x_{61}	&	0	&	0	&x_{64}	&	0  	&0		&0		&x_{68}	&x_{69}	&x_{6,10}	&0		\\ 
\end{block}
\end{blockarray}
 \]
Not only is $I_P\subseteq I_M$ but in this case $I_P$ is the elimination ideal given by eliminating the variables in the columns indexed by the additional hyperplanes $H_6,\ldots,H_{11}$.
%\begin{figure}
%\begin{center}
%\includegraphics[width=3in]{toblerone}
%\caption{The triangular prism.}
%\label{FIG:toblerone}
%\end{center}
% \end{figure}

 \end{example}

%==============================================================================
\section{Realization space models}
\label{SEC:realsp}

A realization space for a rank $d+1$ matroid $M$ with $n$ elements is, roughly speaking, a space whose points are in correspondence with (equivalence classes of) collections of vectors  $V = \{\vv_1,\ldots, \vv_n\}\subset\kk^{d+1}$ whose matroid $M[V]$ is $M$. In this section we show that the slack variety defined in the \S2 provides such a realization space, and we relate it to another realization space called the Grassmannian of the matroid.

Theorem~\ref{THM:slackconditions} characterizes the slack matrices of realizations of a matroid. The next theorem shows that the slack variety captures exactly these matrices. 

\begin{theorem} Let $M$ be a rank $d+1$ matroid. Then $V$ is a realization of $M$ if and only if $S_{M[V]} = S_M(\ss)$ where $\ss \in \VV(I_M)\cap(\kk^*)^t$.  
\label{THM:realizvariety}
%The slack variety is exactly the set of matrices satisfying conditions (i) and (ii) of Theorem~\ref{THM:slackconditions}.
\end{theorem}

\begin{proof} Let $V$ be a realization of $M$. Then $S_{M[V]} = S_M(\ss)$ for some $\ss \in (\kk^*)^t$ by Theorem~\ref{THM:slackconditions} (i). Furthermore, 
$\rk(S_{M[V]}) = d+1$ by Corollary~\ref{COR:rank}, so that its $(d+2)$-minors vanish and thus $\ss\in\VV(I_M)$, as desired. %Since $V$ is a realization, it satisfies condition (i) of  Theorem~\ref{THM:slackconditions}. This means there exists $\ss\in(\kk^*)^t$ such that $S_{M[V]} = S_M(\ss)$, and $\ss \in \VV(I_M)\cap(\kk^*)^t$.

Let $V \in\kk^{(d+1)\times n}$ be such that $S_{M[V]}=S_M(\ss)$ for some $\ss\in\VV(I_M)\cap(\kk^*)^t$. Then $\supp(S_{M[V]}) = \supp(S_M)$ and $\rk(S_{M[V]})\leq d+1$. But now by Lemma~\ref{LEM:trianglesubmat}, $\rk(S_{M[V]})\geq d+1$, making $V$ a realization of $M$ by Theorem~\ref{THM:slackconditions}.
\end{proof}

Since we know that the set of realizations of a matroid is closed under row and column scalings, Theorem~\ref{THM:realizvariety} implies the following corollary. We denote the torus of row and column scalings, $(\kk^*)^n\times(\kk^*)^h$, by $T_{n,h}$. 

\begin{corollary}The slack variety is closed under the action of the group $T_{n,h}$, where $(\kk^*)^n$ acts by row scaling (left multiplication by diagonal matrices) and $(\kk^*)^h$ acts by column scaling (right multiplication by diagonal matrices). \label{COR:scaleclosed}\end{corollary}

Theorem~\ref{THM:realizvariety} and Corollary~\ref{COR:scaleclosed} tell us that the slack variety is a realization space for matroid $M$ and the slack variety modulo the action of $T_{n,h}$ is a realization space for the projective equivalence classes of realizations of $M$.

\begin{definition} \label{DEFN:slackreal} The \emph{slack realization space} of a rank $d+1$ matroid $M$ on $n$ elements with $h$ hyperplanes is the image of the slack variety inside a product of projective spaces 
$\VV(I_M)\cap(\kk^*)^t \hookrightarrow (\mathbb{P}^{n-1})^h,
$ 
where $\ss$ is sent to the columns of $S_M(\ss)$.
\end{definition}

\begin{prop}
Let $M$ be a rank $d+1$ matroid on $n$ elements with $h$ hyperplanes. Then the points of its slack realization space are in one-to-one correspondence with linear equivalence classes of realizations of $M$.
\end{prop}

\begin{proof}
Under this embedding, two slack matrices which differ by column scaling are the same point in $(\mathbb{P}^{n-1})^h$. So, the result follows by Lemma~\ref{LEM:columnscale}.
\end{proof}

Next we describe a known model for the realization space of a matroid arising from a subvariety of a Grassmannian. 
The \emph{Grassmannian} $Gr(d+1,n)$ is a variety whose points correspond to ${(d+1)}$-dimensional linear subspaces of a fixed $n$-dimensional vector space $\Lambda$. %The Grassmannian is a smooth, projective variety of dimension $(d+1)(n-d-1)$, and 
It embeds into $\PP^{{n \choose d+1}-1}$ as follows. Any $(d+1)$-dimensional linear subspace of $\Lambda$ can be described as the row space of a $(d+1)\times n$ matrix of rank $d+1$. However, two matrices $A$ and $B$ have the same row space when then there is a matrix $G \in GL(d+1,\Lambda)$ such that $A = GB$. Thus, to ensure that we have a one-to-one correspondence between subspaces and points in the Grassmannian, we record a $(d+1) \times n$ matrix by its vector of $(d+1)$-minors. We call this the \emph{Pl\"{u}cker vector}, and it has coordinates indexed by subsets $\sigma$ of $[n]$ of size $d+1$.
% [[Old version of this paragraph, deleted and shortened by Maddie ]]
%Any $(d+1)$-dimensional linear subspace of $\Lambda$ can be described as the row space of a $(d+1)\times n$ matrix of rank $d+1$. However, there are multiple ways to assign a matrix to a subspace in this way, since many matrices have the same row space. To ensure that we have a one-to-one correspondence between subspaces and points in the Grassmannian, we record a $(d+1) \times n$ matrix by its vector of $(d+1)$ minors. We call this the \emph{Pl\"{u}cker vector}, and it has coordinates indexed by subsets $\sigma$ of $[n]$ of size $d+1$. Indeed, if two matrices $A$ and $B$ have the same row space, then there is an element $G \in GL(d+1,\Lambda)$ such that $A = GB$, and the Pl\"{u}cker vector of $A$ is $\text{det}(G)$ times the Pl\"{u}cker vector of $B$, so these represent the same point in $\PP^{{n \choose d+1}-1}$; thus we get an embedding of the Grassmannian in projective space. 
The \emph{Pl\"{u}cker ideal} $P_{d+1,n} \subseteq \kk[\pp] = \kk[p_\sigma\ |\ \sigma \subset [n],\ |\sigma| = d+1]$ is the set of all polynomials which vanish on every vector of $(d+1)$-minors coming from some $(d+1)\times n$ matrix. It is generated by the homogeneous quadratic Pl\"ucker relations and cuts out the Grassmannian as a variety inside $\PP^{{n \choose d+1} -1}$.

% It is generated by the Pl\"ucker relations, which are defined as follows. Fix a subsets $I,J\subset[n]$ with $|I| = d$, $|J| = d+2$. For $j\in J$ denote by $\sgn(j,I,J)$ the sign $(-1)^\ell$, where $\ell$ is the number of elements $j' \in J$ with $j<j'$ plus the number of elements of $I$ that are less than $j$.
% Then the Pl\"ucker relation $\mathcal{P}_{I,J}$ is the homogeneous quadratic equation 
%$$\mathcal{P}_{I,J} = \sum_{j\in J} \sgn(j,I,J)\cdot p_{I\cup j}\cdot p_{J\backslash j}.$$
%Note that $\mathcal{P}_{I,J} = 0$ if $|J\backslash I|<3$. If $|J\backslash I|=3$, then after appropriate adjustments we can write $I = I'\cup\{i\}$ and $J = I'\cup\{j,k,\ell\}$ with $i<j<k<\ell$, so that 
%$$\mathcal{P}_{I,J} = p_{I'ij}\cdot p_{I'k\ell} - p_{I'ik}\cdot p_{I'j\ell} + p_{I'i\ell}\cdot p_{I'jk},$$
%where, for example, by $I'ij$ we mean $I'\cup\{i\}\cup\{j\}$. 

If we have a rank $d+1$ matroid $M=(E,\BB)$ with realization $V\in\kk^{(d+1)\times n}$, the Pl\"{u}cker coordinate $p_\sigma$ of $V$ is zero if and only if $\sigma$ is a dependent set of $M$. 
Thus, realizations of $M$ correspond to the subvariety of $Gr(d+1,n)$ defined by setting the Pl\"ucker coordinates of non-bases to 0. This variety is also cut out by an ideal, namely, 
%Then, we may consider realizations of $M$ as a subvariety of $Gr(d+1,n)$. To do this, we define the ring
%$
%\kk[\pp_\mathcal{B}] := \kk[p_\sigma\ |\ \sigma \text{ is a basis of $M$}],
%$
%where we have made one variable $p_\sigma$ for each basis $\sigma \in \mathcal{B}$. 
%Then we define 
the \emph{Grassmannian ideal} $P_M \subset \kk[\pp_\mathcal{B}] $ of $M$,
$$
P_M := (P_{d+1,n} + \langle p_\sigma \ :\ \sigma \not \in \mathcal{B} \rangle) \cap \kk[\pp_\mathcal{B}],
$$
where $\kk[\pp_\mathcal{B}] := \kk[p_\sigma\ |\ \sigma \text{ is a basis of $M$}]$ is the ring with one variable for each basis $\sigma\in\BB$. 
This is the ideal obtained from $P_{d+1,n} \subset \kk[\pp]$ by setting the variables indexed by non-bases of $M$ to 0.

%Let $\mathbb{P}^{|\mathcal{B}|-1}$ be the projective space with coordinates $p_\sigma$ for $\sigma \in \mathcal{B}$, and write $T^{|\mathcal{B}|-1}$ for the points $\mathbf{p}$ whose coordinates are all nonzero. Then we have the following definition.

\begin{definition}\label{DEFN:Grassmannian} The \emph{Grassmannian of $M$}, denoted $Gr(M)$, is $\VV(P_M) \cap (\kk^*)^{|\mathcal{B}|}.$ The points in $Gr(M)$ correspond to $GL(\kk^{d+1})$ equivalence classes of $(d+1)\times n$ matrices which realize the matroid $M$. 
\end{definition}

%We could also interpret $Gr(M)$ as the variety of all $(d+1)$-dimensional linear subspaces of $\kk^n$ whose non-zero Pl\"{u}cker coordinates are precisely the ones indexed by bases of $M$. 
%For example, this variety is empty if and only if $M$ is not realizable over $\kk$.

\subsection{Universal realization ideal}
Given a matroid $M$, we now define an ideal whose variety contains pairs $(\ss,\qq)$, where $\qq$ is a Pl\"{u}cker vector and $\ss$ the nonzero entries of a slack matrix, and both come from the same realization of $M$.

If $V$ is a realization of a rank $d+1$ matroid $M = (E,\BB)$, then a slack matrix $S_{M[V]}$ can be filled with the Pl\"ucker coordinates of $V$, which can be seen from \eqref{EQ:pluckerdet}. Given a hyperplane $H_j\in\HH(M)$, we record all possible substitutions of Pl\"ucker variables for slack variables using a matrix $M_{H_j}$ whose rows are indexed by $i\in E\backslash H_j$, and whose columns are indexed by subsets $J=\{j_1,\ldots, j_d\} \subset E$ with $\overline{J\,} = H$; that is,
$$M_{H_j} = \begin{bmatrix}
x_{i_1j} 	& \sgn(i_1,J_1)\cdot p_{i_1\cup J_1}	 & \cdots 	&  \sgn(i_1,J_k)\cdot p_{i_1\cup J_k}\\
\vdots 	&  \vdots 						 & \cdots 	& \vdots \\
x_{i_mj}	& \sgn(i_m,J_1)\cdot p_{i_m\cup J_1}& \cdots 	& \sgn(i_m,J_k)\cdot p_{i_m\cup J_k}
\end{bmatrix},
$$
where $\sgn(i,J)$ is the sign of the permutation putting $(i,j_1,\ldots,j_d)$ in increasing order.

\begin{example} 
\label{EG:four_lines2}
Recall the matroid $M_4$ of Example \ref{EG:four_lines} pictured in Figure \ref{FIG:four_lines}. Consider the hyperplane $H_2 = 246$. It corresponds to slack variables $x_{i,2}$ for $i=1,3,5$ and its independent subsets are $24$, $26$, and $46$. So the matrix $M_{246}$ has the form 
$$
M_{246}=
\left[\begin{array}{rrrr}
x_{1,2} & p_{124} & p_{126} & p_{146} \\
x_{3,2} & -p_{234} & -p_{236} & p_{346} \\
x_{5,2} & p_{245} & -p_{256} & -p_{456} \\
\end{array}\right].
$$
\end{example}

\begin{definition} \label{DEFN:mother}
Let $M=(E,\BB)$ be a matroid, $P_M$ be the Grassmannian ideal of $M$ and $I_2(M_{H_j})$ be the ideal of $2$-minors of the matrix $M_{H_j}$.
The \emph{universal realization ideal} of $M$ is 
$$U_M := P_M + \sum_{H_j \in \mathcal{H}(M)} I_2(M_{H_j}) \subseteq \kk[\mathbf{x},\mathbf{p]}.$$ 
 
\end{definition}
Intuitively, insisting that the matrices $M_{H_j}$ have rank 1 corresponds to ensuring the columns of the slack matrix are simply scaled versions of the appropriate Pl\"ucker coordinates. 
We now state the main result of this section. 
%Intuitively, the matrices $M_{H_j}$ have columns corresponding to acceptable ways to substitute Pl\"ucker variables for slack variables. The reason we can make these substitutions is that one way to define the hyperplane normals is by taking determinants of subsets of the points, see Example \ref{EG:four_lines} for an explicit calculation. Adding the $2$-minors of the matrices $M_H$ to the ideal $P_M$ corresponds to requiring that these matrices have rank 1, so projectively these columns are all the same vector.
%We now wish to show that this ideal defines a space whose points consist of pairs $(\ss,\qq)$, where $\ss$ defines a slack matrix and $\qq$ defines a Pl\"ucker vector for the same realization of $M$. This result is stated explicitly in Theorem~\ref{THM:mothervariety}, but first we require several preliminary lemmas. 
\begin{theorem} Let $M = ([n],\BB)$ be a rank $d+1$ matroid with universal realization ideal $U_M\subseteq\kk[\xx,\pp]$. Then $\VV(U_M)\in \kk^t\times\kk^{|\BB|}$ with
\begin{enumerate}[label=(\roman{enumi})]
\item the projection of $\VV(U_M)$ onto the Pl\"ucker coordinates, $\pi_\pp:\VV(U_M)\to \kk^{|\BB|},$ is the Grassmannian of the matroid, 
$$\overline{\pi_\pp(\VV(U_M))} = \overline{Gr(M)};$$
\item the projection of $\VV(U_M)\cap\left( (\kk^*)^t\times (\kk^*)^{|\BB|}\right)$ onto the slack coordinates, $\pi_\xx: \VV(U_M)\to \kk^t,$ is the set of slack matrices of realizations of $M$, 
$$\pi_\xx\left(\VV(U_M)\cap \left( (\kk^*)^t\times (\kk^*)^{|\BB|}\right)\right) = \VV(I_M)\cap(\kk^*)^t.$$
\end{enumerate}
\label{THM:mothervariety}\end{theorem}

The proof of this theorem requires several preliminary lemmas. 
We first have the following result on Gr\"obner bases. (For notation and further details see \cite{CLOS}.)

%\comment{Bernd said something here was not correct as stated?? But what?}
\begin{lemma} Fix an elimination order on $\kk[\xx,\pp]$. Given two $\xx$-homogeneous polynomials $f,g$ and an $\xx$-homogeneous set $\mathcal{G}\subset\kk[\xx,\pp]$,  if $h:=\overline{S(f,g)}^{\mathcal{G}}$ with $h\neq 0$, then $\deg_\xx(h) \geq \max\{\deg_\xx(f),\deg_\xx(g)\}$. 
%\label{LEM:xdegree} 
\label{LEM:sausage} \end{lemma}

%\begin{proof} 
%First consider $S(f,g) = \frac{lt(f)}{\gcd(lt(f),lt(g))}\cdot g - \frac{lt(g)}{\gcd(lt(f),lt(g))}\cdot f$. Notice that since $f,g$ are both homogeneous in $\xx$, the S-pair $S(f,g)$ is also homogeneous in $\xx$. Furthermore, the $\xx$-degree $k$ of $\gcd(lt(f),lt(g))$ is at most the minimum of the $\xx$-degrees of $f$ and $g$, so that  
%\begin{align*}
%\deg(S(f,g))	&=\deg_\xx\left(\frac{lt(f)}{\gcd(lt(f),lt(g))}\cdot g\right)\\ 
%			& = \deg_\xx\left(\frac{lt(g)}{\gcd(lt(f),lt(g))}\cdot f\right) \\
%			& = \deg_\xx(f) - k + \deg_\xx(g) \\
%			& \geq  \deg_\xx(f) - \min(\deg_\xx(f),\deg_\xx(g)) + \deg_\xx(g) \\
%			& =  \max(\deg_\xx(f),\deg_\xx(g)).
%\end{align*}
%
%It remains to show that the division algorithm does not decrease the $\xx$-degree of $S(f,g)$. Suppose that the leading monomial of $S(f,g)$ is divisible by the leading monomial of some $p\in\mathcal{G}$. Then Buchberger's algorithm tells us to replace $S(f,g)$ by $$r = S(f,g) - \frac{lt(S(f,g))}{lt(p)}\cdot p.$$
%Since $p$ is also homogeneous in $\xx$, the $\xx$-degree of $r$ is equal to the $\xx$ degree of $S(f,g)$ by construction, as long as $r\neq 0$. Repeating this process on $r$ for all $p\in\mathcal{G}$, we obtain $h:=\overline{S(f,g)}^{\mathcal{G}}$, where if $h\neq 0$ then 
%$$\deg_\xx(h) = \deg_\xx(S(f,g)) \geq \max(\deg_\xx(f),\deg_\xx(g)),$$ as claimed. 
%\end{proof}

\begin{lemma} The Grassmannian ideal of a matroid can be obtained by eliminating the slack variables from its universal realization ideal. That is,
$P_M= U_M \cap \kk[\mathbf{p}].$
\label{LEM:pluckerproj}
\end{lemma}
\begin{proof}
We obtain one containment by the definition of $U_M$, since
$$
P_M  =  P_M \cap \kk[\mathbf{p}] 
	 \subseteq  \left(P_M+\sum_{H\in\HH(M)} I_2(M_{H})\right) \cap \kk[\mathbf{p}] 
     = U_M\cap \kk[\mathbf{p}].
$$
%\begin{eqnarray*}
%P_M & = & P_M \cap \kk[\mathbf{p}] \\
%	& \subseteq & \left(P_M+\sum_{H\in\HH(M)} I_2(M_{H})\right) \cap \kk[\mathbf{p}] \\
%    & =& U_M\cap \kk[\mathbf{p}].
%\end{eqnarray*}
It remains to show the reverse containment. %That is, we need to show if $f+g\in \kk[\mathbf{p}]$ for $f\in P_M$, $g\in I_2(M_H)$ 
Fix an elimination order on $\kk[\xx,\pp]$ eliminating the $\xx$ variables. Let $\mathcal{G}$ be a Gr\"{o}bner basis for $U_M$ with respect to this ordering. Then, it suffices to show that 
\begin{equation}
\label{EQN:buchberger}
\mathcal{G} \cap \kk[\pp] \subset P_M.
\end{equation}
%We will show that all the elements of a Gr\"obner basis for $U_M$, in some elimination order for eliminating the $\xx$ variables, which are in $\kk[\pp]$ are already contained in $P_M$.
% So during Buchberger's algorithm, we add S-pairs of the form $S(f,g)$ for $f = P_{I,J},g=P_{I',J'}$ and $f=P_{I,J}, g=x_ip_j-x_jp_i$ and $f=x_ip_j - x_jp_i, g = x_kp_\ell-x_\ell p_k$. 
If we start with a generating set satisfying \eqref{EQN:buchberger}, then by 
by Lemma~\ref{LEM:sausage}, any terms which are added to $\mathcal{G}$ after applying each step of Buchberger's algorithm with $\xx$-degree 0 must be the reduction of an $S$-pair of elements which also have $\xx$-degree 0, and are therefore also contained in $P_M$.

It remains to show that an initial generating set of $U_M$ satisfies \eqref{EQN:buchberger}. Taking the generating set of the definition, it is enough to show that any minor in $I_2(M_H)$ not containing a slack variable is already in $P_M$. It is not hard to check that any such minor already arises in $P_M$ as some 3-term Pl\"ucker relation having a term $p_\sigma p_\tau$ for some $\sigma\notin\BB$ which gets set to zero. 
%To see this, suppose that $H = \overline{I'\cup\{i,j\}}$ for some $I'\subset[n], |I'|=d-2$, and let $k,\ell\in[n]$ be elements not in the hyperplane $H$, so that there is a minor of $M_H$ of the form %(modulo some signs I should figure out that will be very similar to $\sgn(j,I,J)$ above)
%$$\left|\begin{array}{cc}
% \sgn(k,I'\cup\{i\})p_{I'ik} &  \sgn(k,I'\cup\{j\})p_{I'jk} \\ 
% \sgn(l,I'\cup\{i\}) p_{I'i\ell} &  \sgn(l,I'\cup\{j\}) p_{I'j\ell} 
% \end{array}\right|.$$
%
%But since $H = \overline{I'\cup\{i,j\}}$ is a rank $d-1$ flat, the set $I'\cup\{i,j\}\notin\BB$, hence $p_{I'ij} = 0$ in $P_M$. So the three-term Pl\"ucker relation (up to signs) is
%$$p_{I'ij}\cdot p_{I'k\ell} - p_{I'ik}\cdot p_{I'j\ell} + p_{I'i\ell}\cdot p_{I'jk} \in P_{d+1,n}$$
%which becomes 
%$$- p_{I'ik}\cdot p_{I'j\ell} + p_{I'i\ell}\cdot p_{I'jk} \in P_M,$$
%as claimed.
\end{proof}

%We are now ready to prove the main theorem. 
%\begin{theorem} Let $M$ be a rank $d+1$ matroid on $n$ elements with universal realization ideal $U_M\subseteq\kk[\xx,\pp]$. Then $\VV(U_M)\in (\kk)^t\times\kk^{|\BB|}$ and
%\begin{enumerate}[label=(\roman{enumi})]
%\item the projection of $\VV(U_M)$ onto the Pl\"ucker coordinates, $\pi_\pp:\VV(U_M)\to \kk^{|\BB|}$ is the Grassmannian of the matroid, 
%$$\overline{\pi_\pp(\VV(U_M))} = \overline{Gr(M)},$$
%\item the projection of $\VV(U_M)\cap\left( (\kk^*)^t\times (\kk^*)^{|\BB|}\right)$ onto the slack coordinates $\pi_\xx: \VV(U_M)\to \kk^t$ is the set of slack matrices of realizations of $M$, 
%$$\pi_\xx\left(\VV(U_M)\cap \left( (\kk^*)^t\times (\kk^*)^{|\BB|}\right)\right) = \VV(I_M)\cap(\kk^*)^t.$$
%\end{enumerate}
%\label{THM:mothervariety}\end{theorem}

\begin{proof}[Proof of Theorem~\ref{THM:mothervariety}]\hfill
\begin{enumerate}[label=(\roman{enumi})]
\item This follows from the definition of $Gr(M)$ and Lemma~\ref{LEM:pluckerproj}.
\item[(ii, $\subset$)]
 Let $(\ss,\qq) \in \VV(U_M)\cap  \left((\kk^*)^t\times (\kk^*)^{|\BB|}\right)$. From $\qq$ we can obtain a $(d+1)\times n$ matrix~$V$ with Pl\"ucker vector whose nonzero coordinates come from~$\qq$. We claim that $V$ is a realization of $M$, so that $S_{M[V]}\in\VV(I_M)\cap(\kk^*)^t$ by Theorem~\ref{THM:realizvariety}. Furthermore, we claim that $S_{M[V]}$ and $S_M(\ss)$ are the same up to column scaling, so that  $\ss\in\VV(I_M)\cap(\kk^*)^t$ by Corollary~\ref{COR:scaleclosed}. 
 %It remains to prove the claims. 
	
	A subset of $d+1$ columns of $V$ is independent if and only if the 
	%determinant of that subset is nonzero; in other words, a set of columns is independent if and only if 
	the corresponding Pl\"ucker coordinate is nonzero. Since $\qq\in(\kk^*)^{|\BB|}$, the nonzero Pl\"ucker coordinates correspond  to bases $B\in \BB$ of $M$. Hence $V$ is a realization of $M$. 
	
%	Now fix a column of $M_H$. It is defined by some subset $J = \{j_1<\cdots<j_d\}\subset H$ which spans $H$, and it has entries $(M_H)_{i,J} = \sgn(i,J)\qq_{J\cup\{i\}}$ for $i\notin H$. Since $\qq$ is the Pl\"ucker vector of $V$ this means that $(M_H)_{i,J} = \det(\vv_i,\vv_{j_1},\ldots, \vv_{j_d})$. Furthermore by definition of $U_M$, column $H$ of $S_M(\ss)$ is a scalar multiple, say $\lambda_H$, of this column of $M_H$. Now use subset $J$ to calculate the entries of $S_{M[V]}$ (as in Example~\ref{EG:four_lines}). Then $S_{M[V]}D_h = S_M(\ss)$, where $D_h$ is the $h\times h$ diagonal matrix with entries $\lambda_H$. 
Fix a hyperplane $H$ and an independent subset $J=\{j_1,\ldots, j_d\}\subset H$. Then the first column of $M_H$ is $(s_{iH})_{i\notin H}^\top$, column $J$ is $(\sgn(i,J)\cdot q_{i\cup J})_{i\notin H}^\top$, and by definition of $U_M$, there exists $\lambda_H\in\kk^*$ such that 
$s_{iH} = \lambda_H\sgn(i,J)\cdot q_{i\cup J}$ for all~$i$. Since $\qq$ is the Pl\"ucker vector of $V$, this gives 
$$s_{iH} = \lambda_H\sgn(i,J)\cdot q_{i\cup J} = \lambda_H\det(\vv_i,\vv_{j_1},\ldots,\vv_{j_d}), \quad \forall i$$
so that using subset $J$ to calculate $\alpha_H$, and hence the entries of $S_{M[V]}$, as in Example~\ref{EG:four_lines}, we see that $S_{M[V]}D_h = S_{M}(\ss)$, for $D_h$ the diagonal matrix with entries $\lambda_H$.
	
%	Next, notice that by definition of $U_M$, a column $H$ of $S_M(\ss)$ is a scalar multiple, say $\lambda_H$, of each column of Pl\"ucker coordinates in the appropriate $M_H$ matrix. That column of $M_H$ is defined by a subset $J$ of size $d$ of elements of $V$ which span $H$, so that its entries are of the form $\det(\vv_i, \vv_{j_1},\ldots, \vv_{j_d})$. Use this subset to calculate the entries of $S_{M[V]}$ in column $H$. Then $S_{M[V]}D_h = S_M(\ss)$, where $D_h$ is the $h\times h$ diagonal matrix with entries $\lambda_H$. 

\item[(ii, $\supset$)] Let $\ss\in\VV(I_M)\cap(\kk^*)^t$. By Theorem~\ref{THM:realizvariety}, $S_M(\ss)$ is the slack matrix of some realization $V$ of $M$. Let $\qq$ be the vector of Pl\"ucker coordinates of $V$. We claim $(\ss,\qq)\in\VV(U_M)\cap  {\left((\kk^*)^t\times (\kk^*)^{|\BB|}\right)}$. To see this it suffices to show that the columns of $M_H$, with entries defined using $(\ss,\qq)$, are scalar multiples of each other for each $H\in\HH(M)$; that is, the $I_2(M_H)$ ideals are satisfied, since the Pl\"ucker coordinates necessarily satisfy the Pl\"ucker ideal equations, and hence the equations of $P_M$.
	
	The $(i,H)$ slack entry $s_{iH}$ is of the form $\det(\vv_i,\ww_1,\ldots, \ww_d)$ for some choice of $\ww_1,\ldots, \ww_d$ which span the hyperplane $H$. Each subsequent column of $M_H$ has entries $\det(\vv_i,\vv_{j_1},\ldots,\vv_{j_d})$, obtained from $\qq$ as above, for $j_1<\cdots<j_d$ spanning $H$. Since each $\vv_{j_k}$ lies on hyperplane the $H$, there is a sequence of elementary column operations that takes the matrix with columns $\vv_i,\ww_1,\ldots, \ww_d$ to the one with columns $\vv_i,\vv_{j_1},\ldots,\vv_{j_d}$ for each $i$. These column operations change the determinant by some scale factor $\lambda\in\kk^*$ for all $i$, so that %and will be the same across all rows of $M_H$ for that column. Thus, 
	each column of $M_H$ is a scalar multiple of the first column of slack entries as required.
\end{enumerate} \vspace{-10pt}
\end{proof}

\begin{corollary} Let $M=(E,\BB)$ be a rank $d+1$ matroid. Then
$$\sqrt{I_M} = \sqrt{U_M:\left(\prod_{i\in E, H\in\HH} x_{iH}\prod_{\sigma\in\BB}p_\sigma\right)^\infty}\bigcap \kk[\xx].$$
\label{COR:motherradical}\end{corollary}

 By universality \cite{mnev}, we do not expect that $I_M$ is radical for every matroid. Thus, Corollary \ref{COR:motherradical} may be the strongest relationship between $I_M$ and $U_M$.

\begin{example} 
\label{EG:four_lines3}
We now continue Example \ref{EG:four_lines}, for which we can verify Theorem~\ref{THM:mothervariety} at the level of ideals. 
%Recall the matroid $M_4$, which is a matroid on 6 elements whose non-bases are $123, 246, 345, 156$.
%This can be visualized as the points of intersection of four lines in the plane, as in Figure \ref{FIG:four_lines}. 
Let $P_{M_4}$ be the Grassmannian ideal of this matroid, which is generated by the Pl\"{u}cker relations with variables $p_{123},p_{246},p_{345},p_{156}$ set to 0. Let $I$ be the ideal which guarantees each of the 7 matrices $M_{123}$, $M_{246}$, $M_{345}$ $M_{156}$, $M_{25}$, $M_{14}$, $M_{36}$ have rank 1; that is, $I = \sum_{H\in\HH(M_4)} I_2(M_H)$. 
Then, the universal realization ideal is $U_M = I + P_{M_4}$. 
In this case, we can compute that $P_M = U_M\cap\kk[\pp]$ and $I_M = U_M:\left(\prod x_{iH}\prod p_\sigma\right)^\infty \cap\kk[\xx]$.
\end{example}

%==============================================================================
\section{Non-realizability}
\label{SEC:nonreal}

In this section we illustrate how the slack ideal can be used to determine matroid realizability over a given field. 
This is a well-studied problem \cite{BLSWZ,weird_bernd_book,mnev} for which a complete characterization is only known in a very limited number of cases. Observe that Theorem~\ref{THM:realizvariety} gives us the following criterion for realizability.

\begin{corollary} \label{COR:varietytest}
A matroid $M$ is realizable over $\kk$ if and only if $\VV(I_M)\cap(\kk^*)^t \not = \emptyset$.
\end{corollary}

We now recast this into a test for realizability in terms of the slack ideal. 

\begin{prop} 
Let $M$ be an abstract matroid and $\kk$ be a field. If the slack ideal $I_M = \langle 1\rangle$ over $\kk$, then $M$ is not realizable over $\kk$. If $\kk$ is algebraically closed and $M$ is not realizable over $\kk$, then $I_M=\langle 1 \rangle$.
\label{PROP:nonreal}
\end{prop}

\begin{proof} 
If $M$ is realizable over $\kk$, then there exists a slack matrix $S_M$ which is an element of the variety $\VV(I_M)$ by Theorem \ref{THM:realizvariety}. Then we cannot have $I_M= \langle1\rangle$. On the other hand, if $I_M \not = \langle1\rangle$, then $\VV(I_M)$ is not empty by the Nullstellensatz, and since $I_M$ is saturated with respect to the product of the variables, $\VV(I_M)$ is not contained entirely in the coordinate hyperplanes.
Therefore, by Theorem~\ref{THM:realizvariety} there is a slack matrix $S_M$, and the rows of $S_M$ give a realization of $M$ by Lemma \ref{LEM:rowrealiz}.
 \end{proof}

\begin{example}
\label{EX:fano}
 Consider the Fano plane $M_F$. It is a rank~3 matroid on 7 elements $E = \{0,1,2,3,4,5,6\}$, depicted in Figure~\ref{FIG:fano} with its symbolic slack matrix.  It is known that $M_F$ is only realizable in characteristic 2.
 Using \texttt{Macaulay2}~\cite{M2}, we find $I_{M_F} = \langle 1 \rangle$ over~$\mathbb{Q}$. So, the Fano plane is not realizable over $\mathbb{Q}$ by Proposition \ref{PROP:nonreal}.

Over $\mathbb{F}_2$, the slack ideal is generated by $126$ binomials of degrees 2,3, and 4. So, setting all variables to 1 in the slack matrix gives the realization of $M_F$ in $\mathbb{F}_2^7$.

\begin{figure}
\begin{center}
\begin{minipage}{0.3 \textwidth}
\includegraphics[height = 2 in]{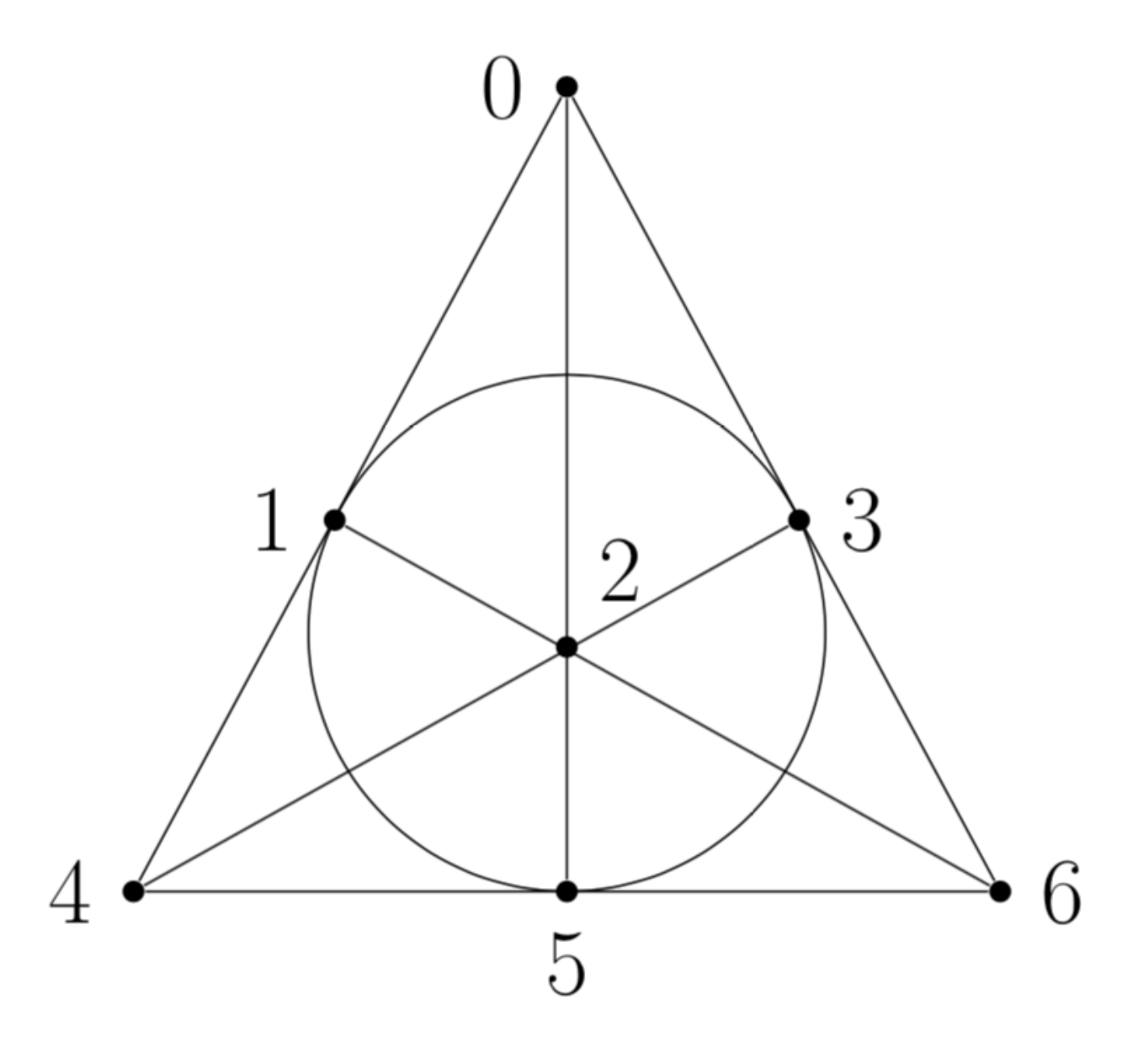}
\end{minipage}
\hspace{0.1 in }
\begin{minipage}{0.65 \textwidth}
\[
\begin{blockarray}{cccccccc}
	&H_1 	& H_2 		& H_3 		& H_4 	& H_5 	& H_6 	& H_7 \\
	&126		&014			&456			&025		&036		&234		&135\\
\begin{block}{c[ ccccccc]}
0&{x}_{01}	&0			&{x}_{03}		&0		&0		&x_{06}	&{x}_{07}\\
1&0			&0			&x_{13}		&{x}_{14}	&{x}_{15}	&x_{16}	&0\\
2&0			&x_{22}		&{x}_{23}		&0		&{x}_{25}	&0		&{x}_{27}\\
3&{x}_{31}	&x_{32}		&{x}_{33}		&{x}_{34}	&0		&0		&0\\
4&x_{41}		&0			&0			&{x}_{44}	&{x}_{45}	&0		&{x}_{47}\\
5&x_{51}		&x_{52}		&0			&0		&x_{55}	&x_{56}	&0\\
6&0			&{x}_{62}		&0			&x_{64}	&0		&x_{66}	&x_{67}\\
\end{block}
\end{blockarray}
 \]
\end{minipage}
\caption{The Fano plane with its non-bases drawn as lines and a circle, and its symbolic slack matrix.}
\label{FIG:fano}
\end{center}
\end{figure}

\end{example}

%\begin{proof} If $M$ has a realization over $\kk$, then its slack matrix is has entries over $\kk$ so that $\VV(I_M)\cap(\kk^*)^t$ contains a point in $\kk$ by Theorem~\ref{THM:realizvariety}.

%Conversely, if $\VV(I_M)\cap(\kk^*)$ contains a point, then by Theorem~\ref{THM:realizvariety} this is a slack matrix of some realization of $M$, the rows of which give a realization of $M$ over $\kk$ by Lemma~\ref{LEM:rowrealiz}.\end{proof}

%\begin{theorem} Let $M$ be a matroid and $I_M \subset \RR[\xx]$ its slack ideal over $\mathbb{R}$. If there exists, $s \in \RR[\xx]$, a sum of squares and $g_1, \ldots, g_k \in \RR[\xx]$ such that $-1=s + g_1 f_1 + \cdots + g_k f_k$ where $f_1, \ldots f_k \in I_M$ then $M$ is not realizable over $\RR$.~\label{THM:realizeC}\end{theorem}
%\begin{proof}
%This is a consequence of applying the real Nullstellensatz \cite{??} to Proposition \ref{PROP:nonreal}.
%\end{proof}

\begin{example}
Consider the \emph{complex matroid} $M_8$ from \cite[p. 33]{weird_bernd_book}. It is a rank~3 matroid on 8 elements with non-bases $124$, $235$, $346$, $457$, $568$, $167$, $278$, and $138$. It is depicted with its symbolic slack matrix in Figure \ref{FIG:complex_matroid}.%In addition to this, it has four two-element hyperplanes $48$, $37$, $26$, $15$ and is realizable over $\CC$ but not $\RR$, as we see below. 
%Its symbolic slack matrix is
%\[
%\begin{blockarray}{ccccccccccccc}
%	&H_1 	& H_2 	& H_3 	& H_4 	& H_5 	& H_6 	& H_7 	&H_8	&H_9 	&H_{10}	&H_{11}	&H_{12}	\\
%	&167		&235		&568		&346		&124		&278		&138		&457		&48		&37		&26		&15\\
%\begin{block}{c[cccccccccccc]}
%1	&0		&x_{12}	&x_{13}	&x_{14}	&0		&x_{16}	&0		&x_{18}	&x_{19}	&x_{1,10}	&x_{1,11}	&0		\\
%2	&x_{21}	&0		&x_{23}	&x_{24}	&0		&0		&x_{27}	&x_{28}	&x_{29}	&x_{2,10}	&0		&x_{2,12}	\\
%3	&x_{31}	&0		&x_{33}	&0		&x_{35}	&x_{36}	&0		&x_{38}	&x_{39}	&0		&x_{3,11}	&x_{3,12}	\\
%4	&x_{41}	&x_{42}	&x_{44}	&0		&0		&x_{46}	&x_{47}	&0		&0		&x_{4,10}	&x_{4,11}	&x_{4,12}	\\
%5	&x_{51}	&0		&0		&x_{54}	&x_{55}	&x_{56}	&x_{57}	&0		&x_{59}	&x_{5,10}	&x_{5,11}	&0		\\
%6	&0		&x_{62}	&0		&0		&x_{65}	&x_{66}	&x_{67}	&x_{68}	&x_{69}	&x_{6,10}	&0		&x_{6,12}	\\
%7	&0		&x_{72}	&x_{73}	&x_{74}	&x_{75}	&0		&x_{77}	&0		&x_{79}	&0		&x_{7,11}	&x_{7,12}	\\
%8	&x_{81}	&x_{82}	&0		&x_{84}	&x_{85}	&0		&0		&x_{88}	&0		&x_{8,10}	&x_{8,11}	&x_{8,12}	\\
%\end{block}
%\end{blockarray}.
% \]

\begin{figure}
\begin{center}
\begin{minipage}{0.26\textwidth}
\includegraphics[height=1.6in]{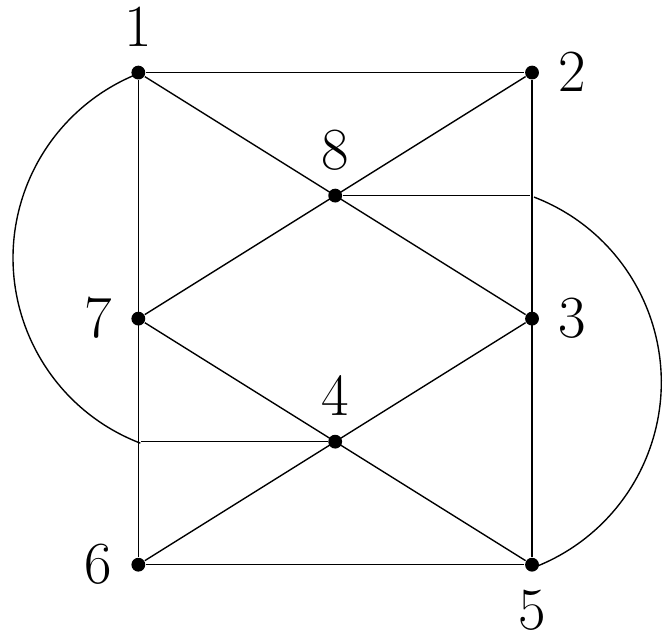}
\end{minipage}\;\;\begin{minipage}{0.69\textwidth}
{\fontsize{8}{10}\selectfont
\[
\begin{blockarray}{*{13}{@{\;\;}c@{\;\;}}}
	&H_1 	& H_2 	& H_3 	& H_4 	& H_5 	& H_6 	& H_7 	&H_8	&H_9 	&H_{10}	&H_{11}	&H_{12}	\\
	&167		&235		&568		&346		&124		&278		&138		&457		&48		&37		&26		&15\\
\begin{block}{c@{\;\;}[*{12}{@{}c@{}}]}
1	&0		&x_{12}	&x_{13}	&x_{14}	&0		&x_{16}	&0		&x_{18}	&x_{19}	&x_{1,10}	&x_{1,11}	&0		\\
2	&x_{21}	&0		&x_{23}	&x_{24}	&0		&0		&x_{27}	&x_{28}	&x_{29}	&x_{2,10}	&0		&x_{2,12}	\\
3	&x_{31}	&0		&x_{33}	&0		&x_{35}	&x_{36}	&0		&x_{38}	&x_{39}	&0		&x_{3,11}	&x_{3,12}	\\
4	&x_{41}	&x_{42}	&x_{44}	&0		&0		&x_{46}	&x_{47}	&0		&0		&x_{4,10}	&x_{4,11}	&x_{4,12}	\\
5	&x_{51}	&0		&0		&x_{54}	&x_{55}	&x_{56}	&x_{57}	&0		&x_{59}	&x_{5,10}	&x_{5,11}	&0		\\
6	&0		&x_{62}	&0		&0		&x_{65}	&x_{66}	&x_{67}	&x_{68}	&x_{69}	&x_{6,10}	&0		&x_{6,12}	\\
7	&0		&x_{72}	&x_{73}	&x_{74}	&x_{75}	&0		&x_{77}	&0		&x_{79}	&0		&x_{7,11}	&x_{7,12}	\\
8	&x_{81}	&x_{82}	&0		&x_{84}	&x_{85}	&0		&0		&x_{88}	&0		&x_{8,10}	&x_{8,11}	&x_{8,12}	\\
\end{block}
\end{blockarray}.
 \]}
\end{minipage}
\caption{The complex matroid $M_8$ and its symbolic slack matrix $S_{M_8}(\xx)$.}
\label{FIG:complex_matroid}
\end{center}
\end{figure}

To simplify the computation, we use Corollary~\ref{COR:scaleclosed} and note that we can select a representative of each projective equivalence class by fixing certain variables in the slack matrix to be 1 (see \S\ref{SSEC:scaledslack} for more details). Fixing the variables $x_{14}$, $x_{27}$, $x_{28}$, $x_{3,12}$, $x_{41}$, $x_{46}$, $x_{47}$, $x_{4,10}$, $x_{57}$, $x_{67}$, $x_{72}$, $x_{73}$, $x_{74}$, $x_{75}$, $x_{77}$, $x_{79}$, $x_{7,11}$, $x_{7,12}$, $x_{84}$ to 1 and computing the slack ideal $I_{M_8}$ in \texttt{Macaulay2} \cite{M2}, we find that it is not the unit ideal. However, it contains the polynomial $x_{8,12}^2+ x_{8,12}+1$. Since this polynomial has only the complex roots $\frac{-1\pm i\sqrt{3}}{2}$, we get by Corollary~\ref{COR:varietytest} that $M_8$ is not realizable over $\RR$, but it is realizable over $\CC$ by Proposition~\ref{PROP:nonreal}.
\end{example}

\subsection{Final Polynomials}

The method of final polynomials introduced in \cite[\S4.2]{weird_bernd_book} certifies when a matroid has no realization. We  define an analogous polynomial for the slack ideal, and show how it can be used to improve computational efficiency of checking non-realizability. 

\begin{definition} \label{DEFN:final} Let $M$ be any matroid of rank $d+1$. Let $\mathcal{S}$ be the multiplicatively closed set generated by taking finite products of the variables $\xx$ in the symbolic slack matrix. A polynomial $f \in \kk[\xx]$ is a \emph{slack final polynomial} if
$$
f \in I_{d+2}(S_M(\xx)) \cap \left(\mathcal{S} + I_M  \right)
$$
where $I_{d+2}(S_M(\xx))$ is the ideal of $(d+2)$-minors of the symbolic slack matrix of $M$. 
\end{definition}

We now have the following result, which demonstrates that the existence of slack final polynomials gives a certificate for non-realizability.

\begin{prop} Let $M$ be a matroid of rank $d+1$. The following are equivalent.
\begin{enumerate}[label = (\roman{enumi})]
\item $1 \in I_M \subseteq \kk[\xx]$,
\item There is a monomial $m \in \kk[\xx]$ such that $m \in I_{d+2}(S_M(\xx))$,
\item A slack final polynomial $f \in  I_{d+2}(S_M(\xx)) \cap \left(\mathcal{S} + I_M  \right)$ exists for $M$.
\end{enumerate}
\label{PROP:finaleq}
\end{prop}

\begin{remark} Over an algebraically closed field, these conditions are equivalent to the matroid being non-realizable by Proposition \ref{PROP:nonreal}. When $\kk$ is not algebraically closed, these conditions imply non-realizability, but if a matroid is not realizable there may not be a slack final polynomial. %\comment{Complex matroid for example?}
\end{remark}

\begin{example} Recall that the complex matroid $M_8$ has a complex realization, as $1\notin I_{M_8} \subseteq \QQ[\xx]$. However, by the above proposition, this means that even though $M_8$ is not realizable over $\QQ$, it does not have a slack final polynomial. 
\end{example}

\begin{proof} \hfill

\begin{enumerate}
\item[(i) $\Rightarrow$ (iii)] Suppose $1 \in I_M$. Since $I_M$ is the saturation of $I_{d+2}(S_M(\xx))$, this implies that there exists a monomial $m \in \kk[\xx]$ such that $m \cdot 1 \in I_{d+2}(S_M(\xx))$. Then, we observe the $m$ is already a slack final polynomial for $M$.

\item[(ii) $\Rightarrow$ (i)] If there is a monomial $m \in \kk[\xx]$ such that $m \in I_{d+2}(S_M(\xx))$, then after saturation we find $1 \in I_M$.

\item[(iii) $\Rightarrow$ (ii)] Suppose $f$ is a slack final polynomial for $M$. Since $f \in (\mathcal{S} + I_M)$, there exists a monomial $m$ and a $g \in I_M$ with $f = m + g$. Since $g \in I_M$, there exists a monomial $n$ such that $ng \in I_{d+2}(S_M(\xx))$, so $nm = nf-ng \in I_{d+2}(S_M(\xx))$ is a monomial in $I_{d+2}(S_M(\xx))$.
\end{enumerate} \vspace{-20pt}
\end{proof}

\begin{remark} In practice saturation of the ideal $I_{d+2}(S_M(\xx))$ can be quite slow, which often makes testing realizability via checking $1\in I_M$ infeasible. Thus the real power of Proposition~\ref{PROP:finaleq} is that one often finds relatively small monomials which are already contained in $I_{d+2}(S_M(\xx))$. So, if one simply wants to certify non-realizability, a faster method is to compute $I_{d+2}(S_M(\xx))$ and check, for example, if $\prod \xx \in I_{d+2}(S_M(\xx))$. In the following example we exhibit how this method can be useful for certifying non-realizability.
\end{remark}

\begin{example} Consider the Fano matroid $M_F$ of Example \ref{EX:fano}. If we compute $I_{4}(S_{M_F}(\xx))$, then we can verify that the product of all of the variables is contained in this ideal. In fact, even the monomial 
$
x_{07}x_{16}x_{25}x_{33}x_{41}x_{52}x_{64}
$
is contained in $I_{4}(S_{M_F}(\xx))$. Verifying this containment (using the laptop of one of the authors) in \texttt{Macaulay2} took 0.000067 seconds, while testing $1\in I_M$ took 3.40765 seconds, indicating a speed up by a factor of 50,000. 
%This gives us a much faster certification that this matroid is non-realizable.
 \end{example}
 
\begin{example} Consider the V\'amos matroid pictured in Figure~\ref{FIG:vamos}. It is a rank~4 matroid $M_v$ on 8 elements whose non-bases are given by the sets $1234$, $1256$, $3456$, $3478$, and $5678$. It is one of the smallest matroids known to be non-realizable over every field. However, the V\'amos matroid has 41 hyperplanes, so that its slack matrix is an $8\times 41$ matrix containing 200 distinct variables. Even computing the full set of minors of this matrix is computationally impractical. 

We note though, that it always suffices to show that Proposition~\ref{PROP:finaleq} (ii) holds for some subideal of the ideal of $(d+2)$-minors. In particular, we can look at the minors of a submatrix of $S_{M_v}(\xx)$. Consider the submatrix of the V\'amos symbolic slack matrix in Figure \ref{FIG:vamos}. 
One can easily check with \texttt{Macaulay2} that the monomial given by the product of all the variables in this submatrix is already in the minor ideal of this submatrix (over $\QQ$ and various finite fields), making $M_v$ non-realizable over these fields by Propositions~\ref{PROP:finaleq} and \ref{PROP:nonreal}. %\comment{over a FIELD!!!}

\begin{figure}
\begin{center}
\begin{minipage}{0.3 \textwidth}
\includegraphics[height=2.5in]{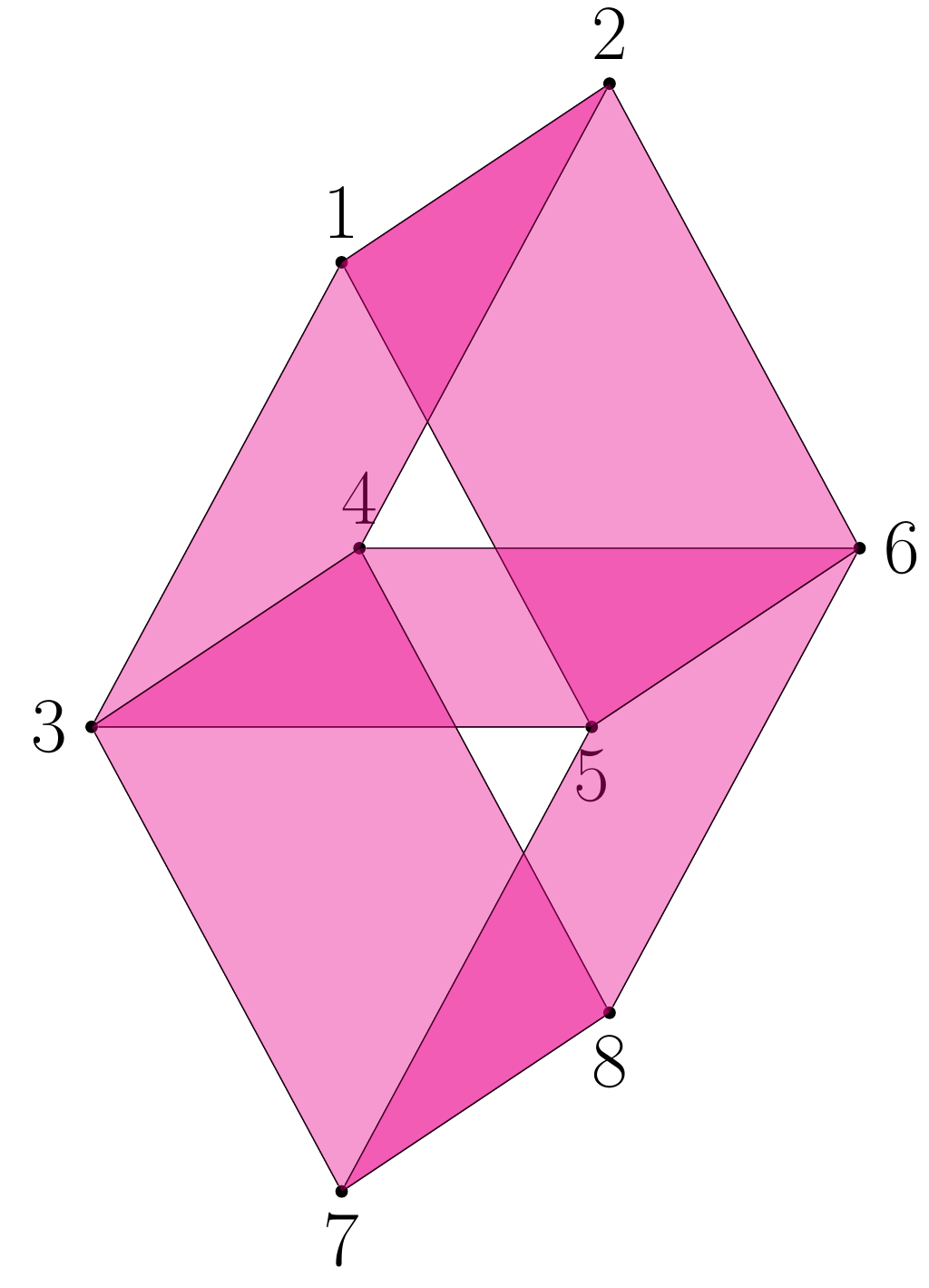}
\end{minipage}
\hspace{0.1 in}
\begin{minipage}{0.65 \textwidth}
\[
\begin{blockarray}{ccccccccc}
	&H_1 	& H_2 	& H_3 	& H_4 	& H_5 	& H_6 	& H_7 	&H_8 \\
	&3456	&567	8	&1234	&3478	&1256	&467		&267		&127 \\
\begin{block}{c[ cccccccc]}
1&	x_{1,1} 	& x_{1,2} 	& 0 		& x_{1,4} 	& 0 		& x_{1,6} 	& x_{1,7}	& 0 \\
2&	x_{2,1} 	& x_{2,2} 	& 0 		& x_{2,4} 	& 0 		& x_{2,6} 	& 0 		& 0 \\
3&	0 		& x_{3,2} 	& 0 		& 0 		& x_{3,5} 	& x_{3,6} 	& x_{3,7} 	& x_{3,8} \\
4&	0 		& x_{4,2} 	& 0 		& 0 		& x_{4,5} 	& 0 		& x_{4,7} 	& x_{4,8} \\
5&	0 		& 0 		& x_{5,3} 	& x_{5,4} 	& 0 		& x_{5,6} 	& x_{5,7} 	& x_{5,8} \\
6&	0 		& 0 		& x_{6,3} 	& x_{6,4} 	& 0 		& 0 		& 0 		& x_{6,8} \\
7&	x_{7,1} 	& 0 		& x_{7,3} 	& 0 		& x_{7,5} 	& 0 		& 0 		& 0 \\
8&	x_{8,1} 	& 0 		& x_{8,3} 	& 0 		& x_{8,5} 	& x_{8,6} 	& x_{8,7} 	& x_{8,8} \\
	\end{block}
	\end{blockarray}
 \]
\end{minipage}
\caption{The V\'amos matroid $M_v$ with non-bases pictured as planes, and a submatrix its slack matrix.}
\label{FIG:vamos}
\end{center}
\end{figure}

\end{example}

\section{Projective uniqueness of matroids}
\label{SEC:toric}

%\comment{What is the best introduction to this section??}
%In general, we do not know whether the slack ideal is prime or even radical. \comment{There should probably be non-prime slack ideals by the universality theorem...right?}

The simplest slack realization spaces are those belonging to projectively unique matroids. In this case, we know that there is a single realization up to projective transformations; in other words, $\VV(I_M)$ is the toric variety which is the closure of the orbit of some realization under the action of $T_{n,h}$. This implies $\sqrt{I_M} = \mathcal{I}(\VV(I_M))$ is a toric ideal; however, universality suggests that $I_M$ need not be radical. A natural question which arises is whether projectively unique matroids correspond exactly to matroids with toric slack ideals. To study this question we introduce an intermediate toric ideal associated to a matroid. %is when a slack ideal is toric. %\comment{Check any of the matroids from the polytope counterexample are also counterexamples in the matroid case.}

\begin{definition}
\label{DEFN:graph}
Define the {\em non-incidence graph of matroid~$M$} as the bipartite graph~$G_M$ with one node for each element of the ground set of $M$, one node for each hyperplane, and an edge between element $i$ and hyperplane $H_j$ if and only if $i\notin H_j$. Notice that $G_M$ records the support of the slack matrix $S_M$, and so we can think of its edges as being labelled by the corresponding entry of $S_M(\xx)$. (See Figure~\ref{FIG:fourlinesgraph} for an example of the graph $G_{M_4}$ for the matroid $M_4$ of Example~\ref{EG:four_lines}, and Figure~\ref{FIG:nonfano_fig} for the non-incidence graph of the non-Fano matroid.)
\end{definition}

%\subsection{The cycle ideal of a matroid}
%\comment{Maybe we should be using cocircuit vocabulary here. Also the cocircuit graph of a matroid is a thing; is it the same or related to this thing?}

Let $\mathcal{A}_G$ be the set of vectors forming the columns of the vertex-edge incidence matrix of the graph $G$, and let $T_G$ be the toric ideal of the vector configuration $\mathcal{A}_G$. The toric ideal of a bipartite graph is a well-studied object \cite{OH99, villarreal}. If a matroid $M$ has a slack matrix $S_M$ which is a 0-1 matrix, then the toric ideal $T_{G_M}$ associated to the graph $G_M$ is the ideal of the orbit of $S_M$ under the action of the torus $T_{n,h}$. So, $T_{G_M}$ describes one projective equivalence class of slack matrices of $M$. We now define an analogous toric ideal for any projective equivalence class. %We now generalize this construction to the case when $S_M$ is any slack matrix. 

\begin{definition} \label{DEFN:cycleideal}
Let $M$ be an abstract matroid with realization $V$. Let $\ss \in (\kk^*)^t$ be such that $S_{M[V]} = S_M(\ss)$, where $t$ is the number of variables in $S_M(\xx)$, the symbolic slack matrix of $M$. We define the \emph{cycle ideal} $C_V$ of $M[V]$ to be the ideal
\begin{equation}
C_V = \left\langle \xx^{c+} - \alpha_c \xx^{c-}  : \text{$c$ is a cycle in $G_M$ and } \alpha_c = \frac{\ss^{c+}}{\ss^{c-}} \right\rangle \subseteq \kk[\xx]
\label{EQ:cycleideal}\end{equation}
where $c+$ and $c-$ are alternating edges from the cycle $c$.
\end{definition}

\begin{theorem} Let $M$ be a matroid of rank $d+1$ on $n$ elements with $h$~hyperplanes. Let $V$ be a realization of $M$ with slack matrix $S_{M[V]} = [s_{i,j}]_{i=1,j=1}^{n,h}$. Then the ideal $C_V$ is the (scaled) toric ideal which is the kernel of the $\kk$-algebra homomorphism 
$
\phi: \kk[\xx] \to \kk[\rr,\tt, \rr^{-1},\tt^{-1}],$
which sends 
$
x_{ij}  \mapsto s_{i,j}r_it_j.
$
\label{THM:cycleistoric}
\end{theorem}

%\begin{proof} Denote by $\aa_{ij}$ the vector with two 1s exactly in positions $i$ and in position $n+j$, so that $(\rr,\tt)^{\aa_{ij}} = r_it_j$. Then the kernel of $\phi$ is $I_{\mathcal{A}}$ the (scaled) toric ideal of the vector configuration $\mathcal{A} = \{\aa_{11},\ldots, \aa_{n,h}\}$. Let $A$ be the matrix with columns $\aa_{ij}$. Notice that the matrix $A$ is the vertex-edge incidence matrix of the graph $G_M$, so that $I_{\mathcal{A}}$ is the (scaled) toric ideal of the graph~$G_M$. Thus it suffices to show that the toric ideal of the graph is generated by binomials  of the form \eqref{EQ:cycleideal}.
%
% One can now adapt any of several proofs (see for example \cite[Proposition 3.1]{villarreal}) regarding the generators of toric ideals of graphs, that the desired kernel is generated by even closed walks in $G_M$. 
%Finally, it is not hard to check that the ideal generated by even closed walks and the ideal generated by cycles are the same; one simply checks that the binomial of an even closed walk is a polynomial combination of the binomials of cycles, since the reverse containment is obvious. 
%\end{proof}

We note that the cycle ideal of a realization provides a way to distinguish projective equivalence classes of realizations of $M$, as well as detect projective uniqueness. 

\begin{lemma}
\label{LEM:cycleproj}
Let $M$ be a matroid with realizations $U$ and $V$. Then $U,V$ are projectively equivalent if and only if $C_V = C_U$. \end{lemma}
\begin{proof}
First suppose that $U$ and $V$ are projectively equivalent. 
Let $S_{M[V]}$ have entries $s_{i,j}$ for elements $i$ of the ground set of $M$ and hyperplanes $H_j$ of $M$. By Lemma~\ref{LEM:pe} we know that  $S_{M[U]}$ is obtained from $S_{M[V]}$ by scaling the rows by $r_1, \ldots, r_n\in\kk^*$ and scaling the columns by $t_1, \ldots, t_h\in\kk^*$. Then the entries of $S_{M[U]}$ are $r_it_js_{i,j}$. Since $S_{M[U]}$ and $S_{M[V]}$ have the same support, they have the same cycles. One may then show via elementary calculation that the coefficients $\alpha_c$ are the same when calculated from $S_{M[U]}$ and from $S_{M[V]}$ for every cycle $c \in G_M$.

%% OLD VERSION-- I think we can cut this, what do you think Amy?
%%%%%%%%%%%%%%%%%%%%%%%%%%%%%%%%%%%%
%So, it is enough to show that the coefficients $\alpha_c$ are the same for every cycle $c \in G_M$. The entries of $S_V$ corresponding to the cycle $c$ will be $s_{{i_1}, {j_1}}, s_{{i_1}, {j_2}},s_{{i_2}, {j_2}}, \ldots, s_{{i_{l}}, {j_l}},s_{{i_l}, {j_1}}$. Then, the coefficient $\alpha_c$ corresponding to this cycle is
%$$
%\alpha_c = \frac{s_{{i_1}, {j_1}}s_{{i_2}, {j_2}}\cdots s_{{i_l}, {j_l}}}{s_{{i_1}, {j_2}}s_{{i_2}, {j_3}}\cdots s_{{i_l}, {j_1}}}.
%$$
%The corresponding coeffiecient in $C_U$ is then
%$$
%\alpha_c' = \frac{(r_{i_1}\cdot t_{j_1}\cdot s_{{i_1}, {j_1}})(r_{i_2}\cdot t_{j_2}\cdot s_{{i_2}, {j_2}})\cdots(r_{i_l}\cdot t_{j_l}\cdot s_{{i_l}, {j_l}})}{(r_{i_1}\cdot t_{j_2}\cdot s_{{i_1}, {j_2}})(r_{i_2}\cdot t_{j_3}\cdot s_{{i_2}, {j_3}})\cdots(r_{i_l}\cdot t_{j_1}\cdot s_{{i_l}, {j_1}})}.
%$$
%So, all of the $r_i$ and $t_j$ cancel, leaving us with $\alpha_c$.

Conversely, suppose $C_V = C_U$. Then $\VV(C_V)=\VV(C_U)$, and in particular, $S_{M[V]} = S_M(\ss)$ and $S_{M[U]} = S_M(\uu)$ for $\ss,\uu\in\VV(C_V)\cap(\kk^*)^t$. We now argue that $S_{M[U]}$ can be row and column scaled to be equal to $S_{M[V]}$, which proves the results by Lemma~\ref{LEM:pe}. 
Fix a spanning forest $T$ of $G_M$. By Lemma \ref{LEM:scaleforest} we may scale the entries in $S_M(\uu)$ corresponding to edges in $T$ to be equal to the corresponding entries of $S_{M[V]}$. %It now suffices to show that the remaining entries are equal. 
Any remaining entry $a$ of $S_{M}(\uu)$ will correspond to an edge $e$ in $G_M$ such that $T \cup \{e\}$ contains a unique cycle $c$, where $e \in c$. The equation $\xx^{c+}-\alpha_c\xx^{c-}\in C_V$ corresponding to this cycle must be satisfied by the scaling of $S_M(\uu)$, since $C_V=C_U$. Since all variables in the cycle except the one labelled by edge $e$ have been fixed, we find that there is only one possible value for $a$. Furthermore, this value equals the corresponding entry in $S_{M[V]}$, since $S_{M[V]}$ satisfies the equations of $C_V$ by definition.  \end{proof}

\begin{corollary} Let $M$ be a matroid with realization $V$. Then, $\VV(C_V) \cap (\kk^*)^t$ consists of exactly the slack matrices of realizations projectively equivalent to $V$.
\label{COR:varofc}
\end{corollary}
%\begin{proof} It is clear by Lemma \ref{LEM:cycleproj} that slack matrices which are projectively equivalent to $S_{M[V]}$ are all contained in $\VV(C_V) \cap (\kk^*)^t$. Let $\ss \in \VV(C_V) \cap (\kk^*)^t$. Then we can make a matrix $S_M(\ss)$ by plugging in the values from the vector $\ss$ into the symbolic slack matrix $S_M$. Now, we argue that $S_M(\ss)$ can be row and column scaled to $S_{M[V]}$. Fix a spanning forest $T$ of $G_M$. By Lemma \ref{LEM:scaleforest} we may scale the entries in $S_M(\ss)$ corresponding to $T$ to be equal to the corresponding entries of $S_{M[V]}$. It now suffices to show that the remaining entries are equal. Any remaining entry $a$ will correspond to an edge $e$ in $G_M$ such that $T \cup \{e\}$ contains a unique cycle $c$, where $e \in c$. Studying the equation for this cycle, we find that there is only one possible value for $a$ since all other variables in the cycle have been fixed, and this value must be equal to the corresponding value in $S_{M[V]}$, since $S_{M[V]}$ satisfies the equations. 
%\end{proof}

\begin{lemma} Let $\kk$ be algebraically closed and $M$ be a matroid with realization $V$. Then the slack ideal $I_M$ is contained in the cycle ideal $C_V$. \label{LEM:cyclecontain}
\end{lemma}
\begin{proof}
By Corollary \ref{COR:varofc} we know that $\VV(C_V) \subset \VV(I_M)$. Then, $\II(\VV(C_V)) \supset \II(\VV(I_M))$. Since $C_V$ is radical, and since  ${I_M}\subset \sqrt{I_M}$, this gives $C_V\supset I_M$.
\end{proof}

\begin{prop} Let $M$ be a matroid with projectively unique realization~$V$. Then $\VV(I_M) = \VV(C_V)$. \label{PROP:PUcyclevar} \end{prop}

\begin{proof} By Corollary~\ref{COR:varofc} and Theorem~\ref{THM:realizvariety}, we get $\VV(I_M)\cap(\kk^*)^t = \VV(C_V)\cap(\kk^*)^t$. Then since both varieties are irreducible, the result follows. 
\end{proof}

In fact, we can have %it can even be the case that the cycle ideal of $M[V]$ can coincide with the slack ideal $I_M$. If 
$I_M = C_V$ for a realization $V$ of $M$. In this case, call $I_M$~{\em cyclic}. 

\begin{theorem} %\comment{Change me, burp me, feed me} 
If the slack ideal of a matroid is cyclic then $M$ is projectively unique and $I_M$ is radical. The converse also holds when $\kk$ is algebraically closed. 
\label{THM:cyclic}
\end{theorem}

\begin{proof} Suppose that $I_M$ is cyclic. Then $M$ is projectively unique by Corollary~\ref{COR:varofc} and $I_M$ is radical, since it is prime by Theorem~\ref{THM:cycleistoric}.
Conversely, suppose that $M$ is projectively unique and $I_M$ is radical. By Proposition~\ref{PROP:PUcyclevar}, $\II(\VV(I_M)) = \II(\VV(C_V))$, so that $I_M = C_V$, as both ideals are radical. 
\end{proof}

%\comment{again, perles might provide an example where we see why the assumption radical in this theorem cannot be dropped.}

\begin{example} Recall the matroid $M_4$ from Example \ref{EG:four_lines}, whose non-incidence graph $G_{M_4}$ is displayed in Figure \ref{FIG:fourlinesgraph}. The matroid $M_4$ has a cyclic slack ideal. Hence $I_{M_4}$ is a radical slack ideal, and $M_4$ is projectively unique.
%By Theorem \ref{THM:cyclic}, we know that $M_4$ is projectively unique and its slack ideal is radical, since the 72 generators of $I_{M_4}$ correspond to the 72 cycles in this graph, which have lengths 4,6,8. 

\begin{figure}
\begin{center}
\includegraphics[height=1.5in]{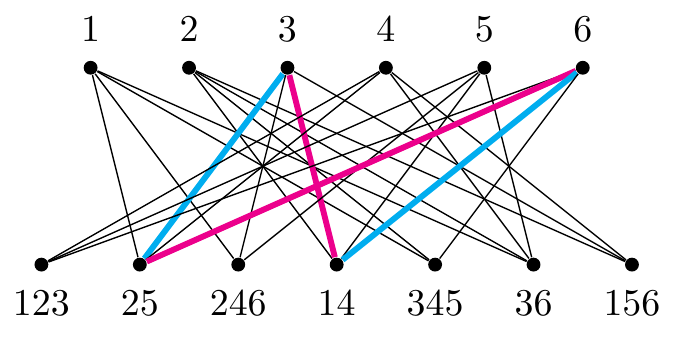}
\end{center}
\caption{The graph $G_{M_4}$ for matroid $M_4$ with the highlighted cycle corresponding to binomial $x_{3,6}x_{6,5}+x_{3,5}x_{6,6}$ of Table~\ref{TAB:72things}.}
\label{FIG:fourlinesgraph}
\end{figure}

\end{example}

\begin{example} Recall the Fano plane discussed in Example \ref{EX:fano}. Over $\mathbb{F}_2$, the 126 binomial generators of $I_{M_F}$ found in Example \ref{EX:fano} correspond to each of the cycles in the graph $G_{M_F}$. Over $\mathbb{F}_2$ this is the projectively unique representation of $M_F$, and this ideal is equal to the cycle ideal of the representation.
\end{example}

%\begin{prop} Does a matroid with $d+2$ ground elements have toric slack ideal equal to $T_M$ as in polytope case? \label{PROP:d+2_toric} \end{prop}

%\begin{example} Something not PU with $I_M\subsetneq T_M$ \end{example}

\subsection{Scaled Slack Matrices} \label{SSEC:scaledslack}

From Corollary~\ref{COR:scaleclosed} we know that quotienting by the action of $T_{n,h}$ on $\VV(I_M)\cap(\kk^*)^t$ gives us a realization space for projective equivalence classes of representations of $M$. We now give an explicit way of computing the variety of these equivalence classes. 
As in \cite[Lemma 5.5]{slack_paper}, we scale rows and columns of a slack matrix via the following lemma, to fix one representative of each projective equivalence class.

\begin{lemma} Given a realization of a matroid $M$, we may scale the rows and columns of its slack matrix $S_M$ so that it has ones in the entries indexed by the edges in a maximal spanning forest $F$ of the graph $G_M$; the resulting realization of $M$ is projectively equivalent to the original realization of $M$. \label{LEM:scaleforest}
\end{lemma}

%\begin{proof}
%\comment{write a proof}
%\end{proof}

%Using this, we can now make a new variety whose points are in one-to-one correspondence with projective equivalence classes of slack matrices.
%By scaling as in Lemma~\ref{LEM:scaleforest}, we are choosing a representative of each projective equivalence class of realizations. 

\begin{definition} Given a matroid $M$ we can take a symbolic slack matrix 
and set variables corresponding to edges in a maximal spanning forest $F$ to 1 
%and scale the rows and columns 
as in Lemma~\ref{LEM:scaleforest} to obtain a \emph{scaled symbolic slack matrix}. Then, the \emph{scaled slack ideal} is obtained by taking the $(d+2)$-minors of this matrix and saturating with respect to the product of all the variables. \label{DEFN:scaled}
\end{definition}
Using the scaled symbolic slack matrix not only allows us to study the projective realization space of $M$, but also proves to be a useful tool for computations because this matrix will have considerably fewer variables.

\begin{example} \label{EX:non-Fano} Let $M_{NF}$ be the non-Fano matroid. It is a rank 3 matroid on 7 elements depicted in Figure~\ref{FIG:nonfano_fig} with its symbolic slack matrix. It differs from the Fano plane by the inclusion of $135$ as a basis. %The hyperplanes are the sets $126$, $014$, $456$, $025$, $036$, $234$, $35$, $15$, $13$. The slack matrix is in Figure \ref{FIG:nonfano_fig}.

 %%% Maddie: These two figures kept awkwardly splitting across the pages because in these two pages we have too many pictures and matrices and not enough text, so I have forced them to all be in one pic here.
 \begin{figure}[h]
\begin{center}
\begin{minipage}{0.25 \textwidth}
\includegraphics[height = 1.6 in]{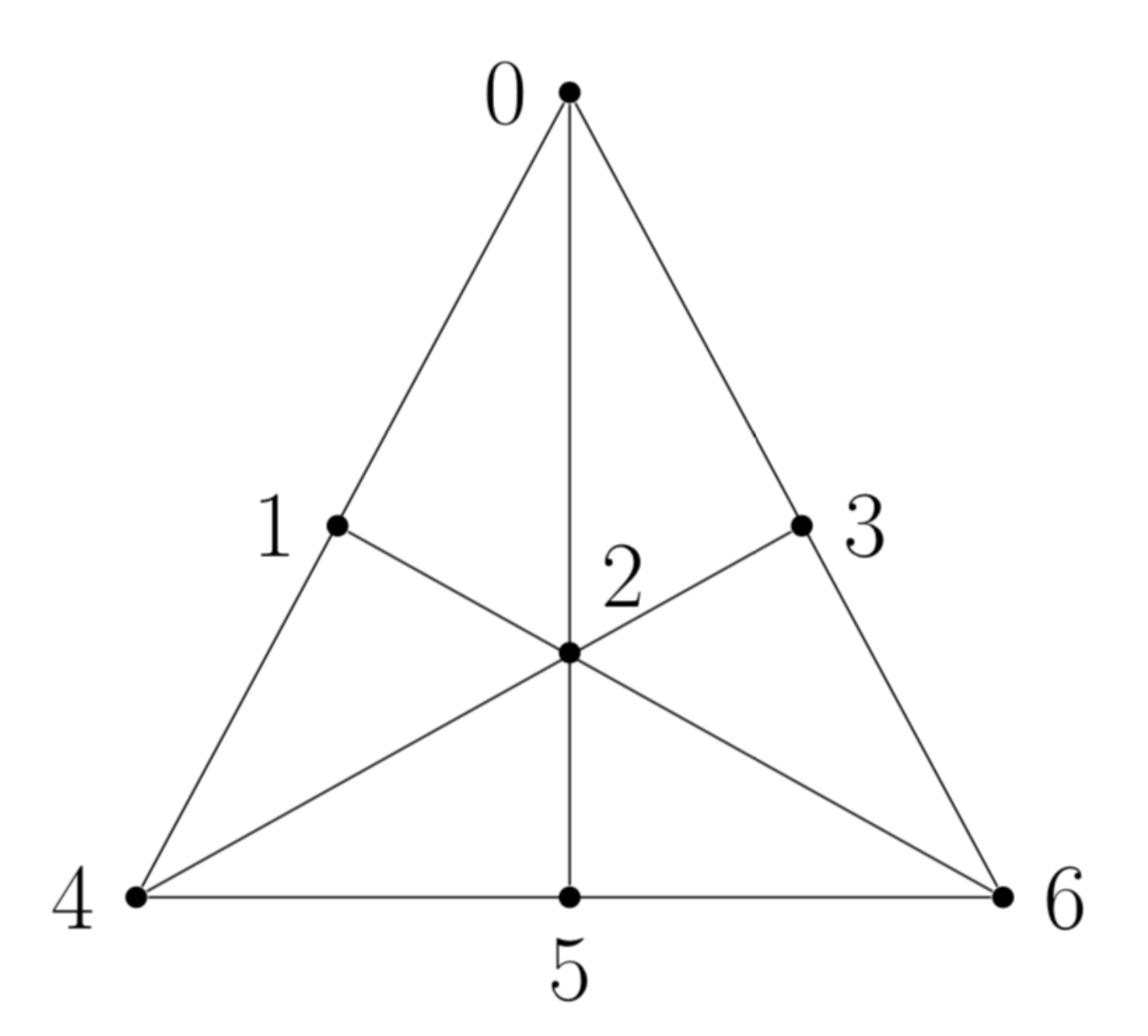}
\end{minipage}
\hspace{0.1 in}
\begin{minipage}{0.7 \textwidth}
\[
\begin{blockarray}{cccccccccc}
	&H_1	& H_2 	& H_3 	& H_4 	& H_5 	& H_6	& H_7		&H_8		&H_9 \\
	&126		&014		&456		&025		&036		&234		&35			&13			&15\\
\begin{block}{c[ccccccccc]}
0	&{x}_{01}	&0		&{x}_{03}	&0		&0		&x_{06}	&{x}_{07}		&x_{08}		&x_{09}\\
1	&0		&0		&x_{13}	&{x}_{14}	&{x}_{15}	&x_{16}	&0			&0			&0\\
2	&0		&x_{22}	&{x}_{23}	&0		&{x}_{25}	&0		&{x}_{27}		&x_{28}		&x_{29}\\
3	&{x}_{31}	&x_{32}	&{x}_{33}	&{x}_{34}&0		&0		&0			&0			&x_{39}\\
4	&x_{41}	&0		&0		&{x}_{44}	&{x}_{45}	&0		&{x}_{47}		&x_{48}		&x_{49}\\
5	&x_{51}	&x_{52}	&0		&0		&x_{55}	&x_{56}	&0			&x_{58}		&0\\
6	&0		&{x}_{62}	&0		&x_{64}	&0		&x_{66}	&x_{67}		&x_{68}		&x_{69}\\
\end{block}
\end{blockarray}
 \]
\end{minipage}

\includegraphics[width = 0.7 \textwidth]{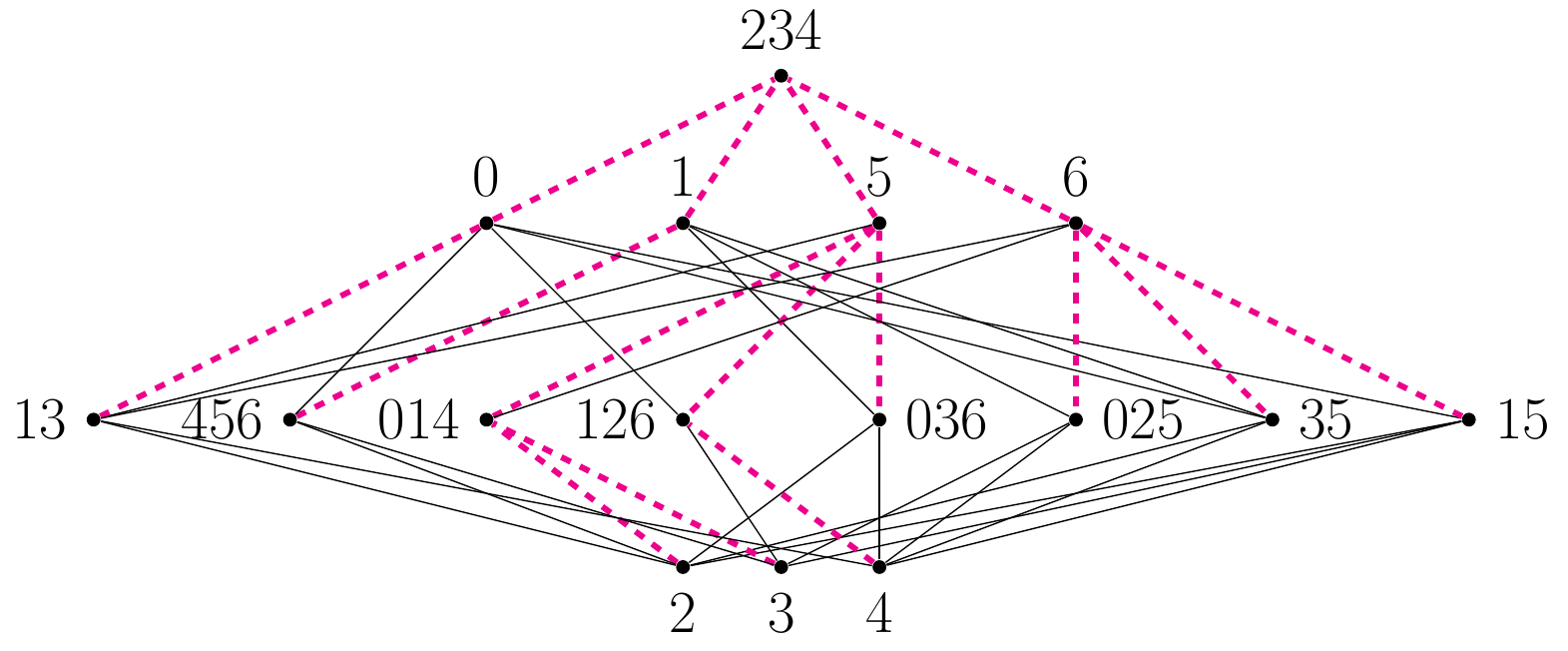}
\end{center}
\caption{The non-Fano matroid, with its non-bases depicted as lines, together with its symbolic slack matrix and the spanning tree $F$ selected of Example~\ref{EX:non-Fano}.}
\label{FIG:nonfano_fig}
\end{figure}

We now show that the non-Fano matroid is projectively unique, and write down a realization from the slack matrix. Let $F$ be the spanning tree of $G_{M_{NF}}$ depicted in Figure \ref{FIG:nonfano_fig}. We set the corresponding variables $x_{41},$ $x_{51},$ $x_{22},$ $x_{32},$ $x_{52},$ $x_{13},$ $x_{64},$ $x_{55},$ $x_{06},$ $x_{16},$ $x_{56},$ $x_{66},$ $x_{67},$ $x_{08},$ $x_{69}$ to 1 in the symbolic slack matrix in Figure \ref{FIG:nonfano_fig}.
%Then we can write the scaled slack matrix as:
%\[
%\begin{blockarray}{cccccccccc}
%	&H_1	& H_2 	& H_3 	& H_4 	& H_5 	& H_6	& H_7		&H_8		&H_9 \\
%	&126		&014		&456		&025		&036		&234		&35			&13			&15\\
%\begin{block}{c[ccccccccc]}
%0	&{x}_{01}	&0		&{x}_{03}	&0		&0		&1		&{x}_{07}		&1			&x_{09}\\
%1	&0		&0		&1		&{x}_{14}	&{x}_{15}	&1		&0			&0			&0\\
%2	&0		&1		&{x}_{23}	&0		&{x}_{25}	&0		&{x}_{27}		&x_{28}		&x_{29}\\
%3	&{x}_{31}	&1		&{x}_{33}	&{x}_{34}&0		&0		&0			&0			&x_{39}\\
%4	&1		&0		&0		&{x}_{44}	&{x}_{45}	&0		&{x}_{47}		&x_{48}		&x_{49}\\
%5	&1		&1		&0		&0		&1		&1		&0			&x_{58}		&0\\
%6	&0		&{x}_{62}	&0		&1		&0		&1		&1			&x_{68}		&1\\
%\end{block}
%\end{blockarray}.
% \]
%  \begin{figure}
%\begin{center}
%\includegraphics[width = 0.8 \textwidth]{nonfanotree}
%\end{center}
%\caption{The spanning tree $F$ selected for the row and column scaling in $S_{M_{NF}}$.}
%\label{FIG:nonfanotree}
%\end{figure}
Taking the ideal of $4$-minors and saturating, we find that the ideal consists of equations of the form $x_{i,j}-\alpha_{i,j}$, for $\alpha_{i,j}\in\QQ$, so the configuration is projectively unique over $\QQ$. The slack matrix corresponding to the single point in $\VV(I_M)\cap(\kk^*)^{39}/T_{7,9}$ is

\scriptsize
$$\bgroup\begin{pmatrix}
1&0&1&0&0&1&1&1&1\\
0&0&1&1&-1&1&2&0&0\\
0&1&-1&0&1&0&-1&-1&1\\
-1&1&-1&1&0&0&0&-2&0\\
1&0&0&-1&1&0&-1&1&1\\
1&1&0&0&1&1&0&0&2\\
0&1&0&1&0&1&1&-1&1\\
\end{pmatrix}\egroup.$$

\end{example}

\begin{example} \label{EG:perles} Consider the \emph{Perles configuration} $M_\star$ of Figure~\ref{FIG:perles}.  
It is a matroid on 9 elements with hyperplanes given by  $0 6 7 8$,  $3 4 7$,  $1 5 6$,  $1 2 8$,  $0 45$,  $3 5 8$,  $0 1 3$,  $2 4 6$,  $2 5 7$, $48$,  $1 7$,  $3 6$,  $1 4$,  $2 3$,  $0 2$. Its symbolic slack matrix is shown in Figure~\ref{FIG:perles} and
%\scriptsize
%\[
%\begin{blockarray}{cccccccccccccccc}
%	&H_1 	& H_2 	& H_3 	& H_4 	& H_5 	& H_6 	& H_7 	&H_8	&H_9 	&H_{10}	&H_{11}	&H_{12}	&H_{13}	&H_{14}	&H_{15}	\\
%	&0 6 7 8	&3 4 7	&1 5 6	&1 2 8	&0 45	&3 5 8	&0 1 3	&2 4 6	&2 5 7	&48		&1 7		&3 6		&1 4		&2 3		&0 2 \\
%\begin{block}{c[ccccccccccccccc]}
%0	&0		&x_{02}	& x_{03}	& x_{04}	&0		& x_{06}	&0		& x_{08}	& x_{09}	& x_{0,10}	& x_{0,11}	& x_{0,12}	& x_{0,13}	& x_{0,14}&0	\\
%1	&x_{11}	& x_{12}	&0		&0		& x_{15}	& x_{16}	&0		& x_{18}	& x_{19}	& x_{1,10}	&0		& x_{1,12}	&0		& x_{1,14}& x_{1,15}\\
%2	&x_{21}	& x_{22}	& x_{23}	&0		& x_{25}	& x_{26}	& x_{27}	&0		&0		& x_{2,10}	& x_{2,11}	& x_{2,12}	& x_{2,13}	&0		&0	\\
%3	&x_{31}	& 0		& x_{33}	& x_{34}	& x_{35}	&0		&0		& x_{38}	& x_{39}	& x_{3,10}	& x_{3,11}	&0		& x_{3,13}	&0		& x_{3,15}\\
%4	&x_{41}	&0		& x_{43}	& x_{44}	&0		& x_{46}	& x_{47}	&0		& x_{49}	&0		& x_{4,11}	& x_{4,12}	&0		& x_{4,14}& x_{4,15}\\
%5	&x_{51}	& x_{52}	&0		& x_{54}	&0		&0		& x_{57}	& x_{58}	&0		& x_{5,10}	& x_{5,11}	& x_{5,12}	& x_{5,13}	& x_{5,14}& x_{5,15}\\
%6	&0		& x_{62}	&0		& x_{64}	& x_{65}	& x_{66}	& x_{67}	&0		& x_{69}	& x_{6,10}& x_{6,11}	&0		& x_{6,13}	& x_{6,14}& x_{6,15}\\
%7	&0		&0		& x_{73}	& x_{74}	& x_{75}	& x_{76}	& x_{77}	& x_{78}	&0		& x_{7,10}	&0		& x_{7,12}	& x_{7,13}	& x_{7,14}& x_{7,15}\\
%8	&0		& x_{82}	& x_{83}	&0		& x_{85}	& 0		& x_{87}	& x_{88}	& x_{89}	&0		& x_{8,11}	& x_{8,12}	& x_{8,13}	& x_{8,14}& x_{8,15}\\
%\end{block}
%\end{blockarray},
% \]
%\normalsize which 
has the matrix $S(\xx)$ studied in \cite[\S4.3]{slack_paper} as a submatrix.

\begin{figure}[h]
\begin{center}
\includegraphics[height=1.7in]{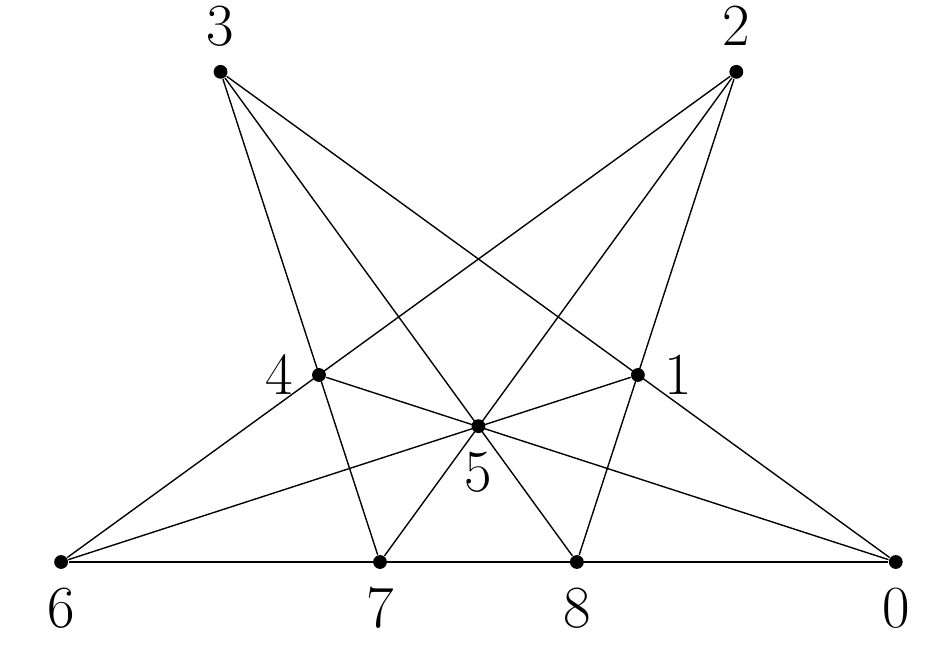}
\scriptsize
\[
\begin{blockarray}{cccccccccccccccc}
	&H_1 	& H_2 	& H_3 	& H_4 	& H_5 	& H_6 	& H_7 	&H_8	&H_9 	&H_{10}	&H_{11}	&H_{12}	&H_{13}	&H_{14}	&H_{15}	\\
	&0 6 7 8	&3 4 7	&1 5 6	&1 2 8	&0 45	&3 5 8	&0 1 3	&2 4 6	&2 5 7	&48		&1 7		&3 6		&1 4		&2 3		&0 2 \\
\begin{block}{c[ccccccccccccccc]}
0	&0		&x_{02}	& x_{03}	& x_{04}	&0		& x_{06}	&0		& x_{08}	& x_{09}	& x_{0,10}	& x_{0,11}	& x_{0,12}	& x_{0,13}	& x_{0,14}&0	\\
1	&x_{11}	& x_{12}	&0		&0		& x_{15}	& x_{16}	&0		& x_{18}	& x_{19}	& x_{1,10}	&0		& x_{1,12}	&0		& x_{1,14}& x_{1,15}\\
2	&x_{21}	& x_{22}	& x_{23}	&0		& x_{25}	& x_{26}	& x_{27}	&0		&0		& x_{2,10}	& x_{2,11}	& x_{2,12}	& x_{2,13}	&0		&0	\\
3	&x_{31}	& 0		& x_{33}	& x_{34}	& x_{35}	&0		&0		& x_{38}	& x_{39}	& x_{3,10}	& x_{3,11}	&0		& x_{3,13}	&0		& x_{3,15}\\
4	&x_{41}	&0		& x_{43}	& x_{44}	&0		& x_{46}	& x_{47}	&0		& x_{49}	&0		& x_{4,11}	& x_{4,12}	&0		& x_{4,14}& x_{4,15}\\
5	&x_{51}	& x_{52}	&0		& x_{54}	&0		&0		& x_{57}	& x_{58}	&0		& x_{5,10}	& x_{5,11}	& x_{5,12}	& x_{5,13}	& x_{5,14}& x_{5,15}\\
6	&0		& x_{62}	&0		& x_{64}	& x_{65}	& x_{66}	& x_{67}	&0		& x_{69}	& x_{6,10}& x_{6,11}	&0		& x_{6,13}	& x_{6,14}& x_{6,15}\\
7	&0		&0		& x_{73}	& x_{74}	& x_{75}	& x_{76}	& x_{77}	& x_{78}	&0		& x_{7,10}	&0		& x_{7,12}	& x_{7,13}	& x_{7,14}& x_{7,15}\\
8	&0		& x_{82}	& x_{83}	&0		& x_{85}	& 0		& x_{87}	& x_{88}	& x_{89}	&0		& x_{8,11}	& x_{8,12}	& x_{8,13}	& x_{8,14}& x_{8,15}\\
\end{block}
\end{blockarray}
 \]
\caption{The Perles configuration matroid $M_\star$ with non-bases shown as lines, and its symbolic slack matrix. }
\label{FIG:perles}
\end{center}
\end{figure}

Since the ideal of $S(\xx)$ will be contained in the ideal of the whole matrix $S_{M_\star}(\xx)$, it follows from computation in \cite{slack_paper} that $M_\star$ is not realizable over $\QQ$. However, it is realizable over $\RR$, and computing its scaled slack ideal we find  that the slack variety consists of the following matrices, where $\alpha$ is a root of the polynomial $\alpha^2-3\alpha+1$:

\tiny
\[
\begin{blockarray}{ccccccccccccccc}
\begin{block}{[ccccccccccccccc]}
0		&1		&\alpha-3	&3-\alpha	&0		&3-\alpha	&0		&-1		&1		&3-\alpha	&1		&\alpha-2	&2-\alpha	&2-\alpha	&0\\
2\alpha-5&1		&0		&0		&3-\alpha	&3-\alpha	&0		&2-\alpha	&3-\alpha	&\alpha-2	&0		&\alpha-2	&0		&2-\alpha	&5-2\alpha\\
1		&\alpha	&1		&0		&\alpha	&\alpha-1	&1		&0		&0		&\alpha	&1-\alpha	&1		&\alpha	&0		&0\\
1		&0		&\alpha-1	&1-\alpha	&\alpha-1	&0		&0		&\alpha-1	&-\alpha	&1		&1-2\alpha&0		&\alpha	&0		&-1\\
2-\alpha	&0		&2-\alpha	&\alpha-1	&0		&\alpha-2	&1		&0		&\alpha-1	&0		&\alpha	&2-\alpha	&0		&\alpha-1&\alpha-1\\
2-\alpha	&1-\alpha	&0		&1		&0		&0		&1		&\alpha-1	&0		&2-\alpha	&1		&1-\alpha	&1		&\alpha	&\alpha-1\\
0		&1		&0		&1		&1		&1		&1		&0		&1		&1		&1		&0		&1		&1		&1\\
0		&0		&3-\alpha	&\alpha-2	&1		&\alpha-2	&1		&1		&0		&\alpha-2	&0		&2-\alpha	&\alpha-1	&\alpha-1	&1\\
0		&\alpha+1	&1		&0		&1		&0		&1		&\alpha	&1-\alpha	&0		&1-\alpha	&1-\alpha	&\alpha	&\alpha&1\\
\end{block}
\end{blockarray}.
 \]
\normalsize

Over any field where $5$ has a square root, the variety has degree 2 and dimension~0, so it consists of two points obtained by setting $\alpha = \frac{3 \pm \sqrt{5}}{2}$. Over $\mathbb{F}_5$, this polynomial has a double root $\alpha=4$. In particular, it is projectively unique here and the scaled slack variety is non-reduced. %We thank \comment{whoever pointed this out to us} for this observation.

\end{example}

%%==============================================================================
%\section{Conclusion}
%\subsection{Questions}
%
%\begin{question}
%Does there exist a non-projectively unique matroid with toric slack ideal? With slack ideal contained but not equal to a cycle ideal?
%\end{question}
%
%\begin{question}
%Can we find an explicit example of a non-radical slack ideal? Is the full slack ideal of Example~\ref{EG:perles} such an example?
%\end{question}
%
%\begin{question}
%What can we say about the structure of slack ideals of uniform matroids?
%\end{question}
%
%\begin{question}
%Does the degree or multidegree of the slack ideal relate to the combinatorics of the matroid?
%\end{question}
%
%\begin{question}
%Is the slack ideal of a polytope always the elimination ideal of the slack ideal of its matroid?
%\end{question}
%
%\begin{question}
%Can Corollary~\ref{COR:motherradical} be strengthened?
%\end{question}
%
%\begin{question}
%Can we identify a smallest monomial to use for testing realizability via Proposition~\ref{PROP:finaleq}? 
%\end{question}

%% MADDIE: I forced this newpage since things were lining up really weirdly
\newpage
\subsection{Acknowledgements}

We would like to thank Bernd Sturmfels and Rekha Thomas for their guidance and helpful discussion. We also thank Dan Corey for pointing us to examples of matroids with interesting realization spaces. 

The software package \texttt{Macaulay2} \cite{M2} was invaluable in calculating all of the examples from this paper. In addition, the package ``Matroids" written for \texttt{Macaulay2} by Justin Chen was indispensable. 

We also acknowledge the Mathematical Sciences Research Institute, the Max-Planck-Institut f\"ur Mathematik in den Naturwissenschaften, and the University of Washington for facilitating our collaboration on this paper. 

\newpage
\appendix

\section{Notation}
\label{AP:notation}
%%% AMY: see what you think. I tried to cut down the number of notation things 

\thispagestyle{empty}
\begin{table}[h]
\begin{center}
\begin{tabular}{| ll|}
\hline
Name 			& Description\\
\hline
$\alpha_H$		& a hyperplane normal to the hyperplane $H$ \\
%$\mathcal{B}$		& bases of $M$\\
$C_V$			& cycle ideal of realization $V$ \\
%$d+1$			& rank of $M$ \\
%$E$ 				& ground set of $M$ \\
$G_M$			& non-incidence graph of $M$ \\
%$Gr(d+1,n)$		& the Grassmannian of $d+1$-subspaces of $n$-dimensional space \\
%$Gr(M)$			& the Grassmannian $M$, see Definition \ref{DEFN:Grassmannian} \\
%$\mathcal{H}(M)$	& hyperplanes of $M$ \\ 
%$H_i$			& a hyperplane of $M$\\
%$h$				& number of hyperplanes of $M$\\
%$\mathcal{I}(M)$	& independent sets of $M$ \\ 
$I_M$			& slack ideal of $M$, see Definition \ref{DEFN:symbslack}\\
%$\mathbb{K}$		& a field\\
%$M$ 				& matroid \\
$M[V]$			& matroid given by realization $V$\\
$M_H$			& Pl\"ucker substitution matrices, see Definition~\ref{DEFN:mother} \\
%$M_4$			& four-lines matroid, see Figure \ref{FIG:four_lines}\\
%$M_F$			& Fano matroid, see Figure~\ref{FIG:fano} \\
%$M_{NF}$			& non-Fano matroid, see Figure~\ref{FIG:nonfano_fig} \\
%$M_8$			& complex matroid, see Figure~\ref{FIG:complex_matroid} \\
%$M_v$			& V\'amos matroid, see Figure~\ref{FIG:vamos} \\
%$M_\star$			& Perles matroid, see Figure~\ref{FIG:perles} \\
%$n$				& number of elements of $E$ \\
%$P$				& triangular prism \\
%$P_{d+1,n}$		& the Pl\"{u}cker ideal \\
%$\mathbf{p}_{\sigma}$& a  Pl\"{u}cker variable\\
%$P_M$			& the Grassmannian ideal of $M$ \\
%$\sgn(j,I)$			& sign of the permutation that puts the elements $j,I_1,\ldots, I_d$ in order \\
$S_M$			& slack matrix of $M$, see Definition \ref{DEFN:slackmatrix}\\
$S_M(\mathbf{x})$	& symbolic slack matrix, see Definition \ref{DEFN:symbslack} \\
$t$				& number of slack variables \\
$T_{n,h}$			& torus of row and column scalings $(\kk^*)^n\times(\kk^*)^h$ \\
$U_M$			& the universal realization ideal of $M$, see Definition \ref{DEFN:mother} \\
%$\mathbf{v}_i$		& vectors, usually realizing $M$\\
$\mathcal{V}(I_M)$	& slack variety of $M$, see Definition \ref{DEFN:symbslack}\\
$x_{ij}$			& slack variables, see Definition \ref{DEFN:symbslack}\\
\hline
\end{tabular}
%\caption{Notation used in this paper.}%, aka the alphabet.}
\end{center}
\end{table}

%%%%%%%%%%
\nocite{*}
\bibliographystyle{amsalpha}
\bibliography{ref}
\end{document}